\documentclass[11pt]{article}

\usepackage[a4paper,height=191mm,width=134mm]{geometry}

\usepackage[utf8]{inputenc}
\usepackage{amsthm}
\usepackage{amssymb}
\usepackage{latexsym}
\usepackage{amsmath}
\usepackage{amsfonts}
\usepackage{mathrsfs}
\usepackage{amstext}
\usepackage{enumitem}
\setitemize{noitemsep,topsep=2pt,parsep=2pt,partopsep=2pt}
\usepackage{mathtools}
\DeclarePairedDelimiter{\ceil}{\lceil}{\rceil}
\DeclarePairedDelimiter{\floor}{\lfloor}{\rfloor}
\mathtoolsset{showonlyrefs}
\usepackage{stmaryrd}
\usepackage{tocloft}
\usepackage{cite}

\makeatletter
\newsavebox{\@brx}
\newcommand{\llangle}[1][]{\savebox{\@brx}{\(\m@th{#1\langle}\)}%
  \mathopen{\copy\@brx\kern-0.6\wd\@brx\usebox{\@brx}}}
\newcommand{\rrangle}[1][]{\savebox{\@brx}{\(\m@th{#1\rangle}\)}%
  \mathclose{\copy\@brx\kern-0.6\wd\@brx\usebox{\@brx}}}
\newcommand{\VERT}[1][]{\savebox{\@brx}{\(\m@th{#1|}\)}%
 \mathopen{\copy\@brx\kern-0.6\wd\@brx\copy\@brx\kern-0.6\wd\@brx\usebox{\@brx}}}
\newcommand{\VERTT}[1][]{\savebox{\@brx}{\(\m@th{#1|}\)}%
 \mathopen{\copy\@brx\kern-0.6\wd\@brx\copy\@brx\kern-0.6\wd\@brx\copy\@brx\kern-0.6\wd\@brx\usebox{\@brx}}} 
\newcommand{\RECT}[1][]{\savebox{\@brx}{\(\m@th{#1[\hspace{-0.3mm}]}\)}}
\makeatother

\newcommand{\bN}{\mathbb{N}} 
\newcommand{\bE}{\mathbb{E}}

\newcommand{\bZ}{\mathbb{Z}}
\newcommand{\bR}{\mathbb{R}}

\newcommand{\bT}{\mathbb{T}}
\newcommand{\bM}{\mathbb{M}}

\newcommand{\cB}{\mathcal{B}}

\newcommand{\cV}{\mathcal{V}}

\newcommand{\cD}{\mathcal{D}}

\newcommand{\cP}{\mathcal{P}}
\newcommand{\cX}{\mathcal{X}}

\newcommand{\cK}{\mathcal{K}}

\newcommand{\fA}{\mathbf{A}}
\newcommand{\fB}{\mathbf{B}}
\newcommand{\fV}{\mathbf{V}}
\newcommand{\fI}{\mathbf{I}}

\newcommand{\fL}{\mathbf{L}}
\newcommand{\fQ}{\mathbf{Q}}
\newcommand{\fP}{\mathbf{P}}

\newcommand{\fT}{\mathbf{T}}
\newcommand{\fU}{\mathbf{U}}
\newcommand{\fX}{\mathbf{X}}
\newcommand{\fY}{\mathbf{Y}}
\newcommand{\fZ}{\mathbf{Z}}

\newcommand{\frM}{\mathfrak{M}}

\newcommand{\sS}{\mathscr{S}}
\newcommand{\sC}{\mathscr{C}}

\newcommand{\ri}{\mathrm{i}}

\newcommand{\rd}{\mathrm{d}}
\newcommand{\rdim}{\mathsf{d}}

\newcommand{\rc}{\mathrm{c}}
\newcommand{\rb}{\mathrm{b}}

\newcommand{\ry}{\mathrm{y}}

\newcommand{\rt}{\mathrm{t}}
\newcommand{\rD}{\mathrm{D}}
\newcommand{\sIR}{\mu}
\newcommand{\uIR}{\eta}

\newcommand{\uv}{\kappa}
\newcommand{\oo}{{\mathtt{g}}}
\newcommand{\ooo}{\mathtt{f}}
\newcommand{\oooo}{\mathtt{h}}
\newcommand{\I}{\mathtt{I}}
\newcommand{\vI}{\mathtt{I}}
\newcommand{\vJ}{\mathtt{J}}
\newcommand{\vK}{\mathtt{K}}
\newcommand{\vL}{\mathtt{L}}
\newcommand{\vM}{\mathtt{M}}
\newcommand{\ro}{\mathtt{l}}

\newcommand{\rri}{i}
\newcommand{\ra}{a}
\newcommand{\rrm}{m}
\newcommand{\rs}{s}
\newcommand{\rr}{r}
\newcommand{\rn}{n}

\makeatletter
\newcommand*\botimes{{\mathpalette\botimes@{1.5}}}
\newcommand*\botimes@[2]{\mathbin{\vcenter{\hbox{\scalebox{#2}{\hspace{0.0mm}$\m@th#1\otimes$\hspace{0.0mm}}}}}}
\newcommand*\Cdot{{\mathpalette\Cdot@{.6}}}
\newcommand*\Cdot@[2]{\mathbin{\vcenter{\hbox{\scalebox{#2}{\hspace{0.5mm}$\m@th#1\bullet$\hspace{0.5mm}}}}}}
\makeatother

\newtheorem{thm}{Theorem}
\newtheorem{lem}[thm]{Lemma}

\theoremstyle{remark}
\newtheorem{rem}[thm]{Remark}

\theoremstyle{remark}

\theoremstyle{definition}
\newtheorem{dfn}[thm]{Definition}

\numberwithin{equation}{section}
\numberwithin{thm}{section}
\numberwithin{example}{section}
\numberwithin{ass}{section}

\usepackage{color,hyperref}
\definecolor{darkblue}{rgb}{0.0,0.0,0.5}
\hypersetup{
  linkcolor  = darkblue,
  citecolor  = darkblue,
  urlcolor   = darkblue,
  colorlinks = true,
}

\makeatletter
\AtBeginDocument{
  \hypersetup{
    pdftitle = {\@title},
    pdfauthor = {\@author}
  }
}
\makeatother

\begin{document}
\title{Renormalization of singular elliptic stochastic PDEs using flow equation}
\author{Pawe{\l} Duch
\\
Faculty of Mathematics and Computer Science \\  
Adam Mickiewicz University in Pozna\'n\\
ul. Uniwersytetu Pozna\'nskiego 4, 61-614 Pozna\'n, Poland\\
pawel.duch@epfl.ch
}
\date{\today}

\maketitle

\begin{abstract}
We develop a solution theory for singular elliptic stochastic PDEs with fractional Laplacian, additive white noise and cubic non-linearity. The method covers the whole sub-critical regime. It is based on the Wilsonian renormalization group theory and the Polchinski flow equation.

\vspace{1mm}

\end{abstract}

\tableofcontents

\section{Introduction}

A general technique that allows to renormalize and prove universality of parabolic singular SPDEs with fractional Laplacian, additive noise and polynomial non-linearity was developed in~\cite{duch}. The goal of this paper is to give a different application of this technique. We present a self-contained construction of solutions of the non-local singular elliptic SPDEs
\begin{equation}\label{eq:spde}
 \big(1+(-\Delta)^{\sigma/2}\big)\varPhi(x) = \xi(x)+\lambda\varPhi(x)^3 - \infty\varPhi(x),
 \quad 
 x\in\bR^\rdim,
\end{equation}
where $\rdim\in\{1,\ldots,6\}$, $\sigma\in(\rdim/3,\rdim/2]$, $(-\Delta)^{\sigma/2}$ is the fractional Laplacian, $\lambda\in\bR$ is sufficiently small and $\xi$ is the periodization of the white noise on $\bR^\rdim$ with period $2\pi$. Recall that the regularity of the noise $\xi$ is slightly worse than $-\rdim/2$ and the expected regularity of the solution is slightly worse than $\sigma-\rdim/2$. For $\sigma>\rdim/2$ the above equation is not singular and can be solved using classical PDE theory. For $\sigma\leq\rdim/2$ the solution is not a function but only a distribution. As a result, the cubic term is ill-defined and has to be renormalized by subtracting an appropriate mass counterterm (for $\rdim>6$ other counterterms are needed even if one takes into account all the symmetries of the equation). The renormalization problem is tractable only if the equation is sub-critical (super-renormalizable). This is the case if the expected regularity of the renormalized non-linearity is better than the regularity of $\xi$, i.e. $3(\sigma-\rdim/2)>-\rdim/2$. Let us remark that for $\rdim=5$ and $\sigma=2$ the above equation is sometimes called the elliptic quantization equation of the $\varPhi^4_3$ model (provided $\xi$ is replaced by the white noise on $\bR^\rdim$). 

Let $G\in L^1(\bR^\rdim)$ be the fundamental solution for the pseudo-differential operator $\fQ:=1+(-\Delta)^{\sigma/2}$ and let $G_\uv$ be the smooth approximation of $G$ with a spatial UV cutoff of order $[\uv]:=\uv^{1/\sigma}$ introduced in Def.~\ref{dfn:kernel_G}. We rewrite the above singular SPDE in the regularized mild form
\begin{equation}\label{eq:intro_mild}
 \varPhi = G_\uv \ast F_\uv[\varPhi],\qquad \uv\in(0,1/2].
\end{equation}
The functional $F_\uv[\varphi]$, called the force, is defined by
\begin{equation}\label{eq:force}
 F_\uv[\varphi](x):= \xi(x) 
 +\lambda\, \varphi^3(x)
 +\sum_{i=1}^{i_\sharp}\lambda^i\, c_{\uv}^{[i]} \varphi(x),
\end{equation}
where $i_\sharp=\floor{\sigma/(3\sigma-\rdim)}$ and the parameters $c_{\uv}^{[i]}\in\bR$ depending on the UV cutoff $\uv$ are called the counterterms. Let us state our main result.

\begin{thm}
There exists a choice of counterterms and a random variable $\lambda_0$ such that $\bE(\lambda_0^{-n})<\infty$ for every $n\in\bN_+$ and for every random variable \mbox{$\lambda\in[-\lambda_0,\lambda_0]$} and $\uv\in(0,1/2]$ Eq.~\eqref{eq:intro_mild} has a periodic solution \mbox{$\varPhi_\uv\in C^\infty(\bR^\rdim)$} and for all $\beta<\sigma-\rdim/2$ the limit $\lim_{\uv\searrow0}\varPhi_\uv$ exists almost surely in the Besov space $\sC^\beta(\bR^\rdim)$. 
\end{thm}
\begin{proof}
We first establish bounds for cumulants of the enhanced noise introduced in Def.~\ref{dfn:modified_coefficients}. The bounds are stated in Theorem~\ref{thm:cumulants} and hold true for an appropriate choice of the counterterms. Using these bounds and a Kolmogorov type argument we deduce bounds for the enhanced noise stated in Theorem~\ref{thm:probabilistic_bounds}. This together with the deterministic result of Theorem~\ref{thm:solution} implies the statement.
\end{proof}

\begin{rem}
There exists a general technique developed in~\cite{hairer2014structures,chandra2016bphz,bruned2019algebraic,bruned2021renormalising} based on the theory of regularity structures that allows to systematically renormalize virtually all sub-critical singular SPDEs with local differential operators. However, due to the slow decay of the kernel $G$ at infinity, this technique does not cover Eq.~\eqref{eq:intro_mild}.
\end{rem}

\begin{rem}
Our technique requires a small parameter. In the case of the elliptic equations we assume that the prefactor of the non-linear term is sufficiently small. Alternatively, one could introduce a prefactor in front of the noise term and assume that it is sufficiently small. In the case of parabolic equations the technique allows~\cite{duch} to construct a solution in a sufficiently small time interval without any assumption about the strength of the non-linearity.
\end{rem}

\begin{rem}
We study Eq.~\eqref{eq:intro_mild} with a regularized Green function $G_\uv$ and the periodization of the white noise $\xi$. Our technique is also applicable~\cite{duch} to the equation with the standard Green function $G$ and a regularized noise $\xi_\uv$.
\end{rem}

\section{Effective force}\label{sec:intro_flow}

The basic object of the flow equation approach is the effective force functional
\begin{equation}
 \sS(\bM)\ni\varphi\mapsto F_{\uv,\sIR}[\varphi]\in \sS'(\bM)
\end{equation}
depending on the UV cutoff $\uv\in(0,1/2]$ and the flow parameter $\sIR\in[0,1]$, where $\sS(\bM)$ and $\sS'(\bM)$ denote the space of Schwartz functions and distributions on $\bM=\bR^\rdim$, respectively. Note that $\kappa$ is assumed to be strictly positive.

\begin{dfn}\label{dfn:kernel_G}
Fix $\chi\in C^\infty(\bR)$ such that $\chi(r)=0$ for $|r|\leq1$ and $\chi(r)=1$ for $|r|>2$ and let $\chi_\sIR(r) := \chi(r(1-\sIR)/\sIR)$ for $\sIR\in(0,1]$. Let $G\in L^1(\bM)$ be the fundamental solution for the pseudo-differential operator $\fQ:=1+(-\Delta)^{\sigma/2}$. For $\uv\in(0,1/2]$ and $\sIR\in(0,1]$ the smooth kernels $G_\uv,G_{\uv,\sIR}\in L^1(\bM)$ are defined by
\begin{equation}
 G_{\uv}(x) := 
 \chi_\uv(|x|^\sigma)\,
 G(x),
 \qquad G_{\uv,\sIR}(x) := 
 \chi_\uv(|x|^\sigma)\,
 \chi_\sIR(|x|^\sigma)\,
 G(x).
\end{equation}
\end{dfn}

\begin{rem}
Since $\chi_\uv(r)\chi_\sIR(r)=\chi_\uv(r)$ for $\sIR\leq\uv/2$ it holds that $G_{\uv,\sIR}=G_\uv$ for $\sIR\leq\uv/2$ and $G_{\uv,\sIR}=G_\sIR$ for $\sIR\geq 2\uv$. Moreover, $\lim_{\uv\searrow0}G_\uv=G$ in $L^1(\bM)$ and $G_{\uv,1}=0$ for all $\uv\in(0,1/2]$. Note that $\uv=1/2$ corresponds to UV cutoff at spatial scale $1$. For arbitrary fixed UV cutoff $\uv\in(0,1/2]$ the family $G_{\uv,\sIR}$, $\sIR\in[0,1]$, interpolates between $G_\uv$ and $0$. Because of slow decay of $G$ for $\sigma\notin2\bN_+$ the range $\sIR\in[1/2,1]$ will require some special treatment.
\end{rem}

By definition, the effective force satisfies the flow equation
\begin{equation}\label{eq:intro_flow_eq}
 \langle\partial_\sIR F_{\uv,\sIR}[\varphi],\psi\rangle
 =
 -\langle \rD F_{\uv,\sIR}[\varphi,\partial_\sIR G_{\uv,\sIR}\ast F_{\uv,\sIR}[\varphi]],\psi\rangle
\end{equation}
with the boundary condition $F_{\uv,0}[\varphi]=F_\uv[\varphi]$, where $\varphi,\psi\in \sS(\bM)$ and the force $F_\uv[\varphi]$ is defined by Eq.~\eqref{eq:force}. The pairing between a distribution $V$ and a test function $\psi$ is denoted by $\langle V,\psi\rangle$ and $\langle\rD V[\varphi,\zeta],\psi\rangle$ is the derivative of the functional $\sS(\bM)\ni\varphi\mapsto \langle V[\varphi],\psi\rangle\in\bR$ in the direction $\zeta\in \sS(\bM)$. In contrast to the force $F_\uv[\varphi]$ the effective force $F_{\uv,\sIR}[\varphi]$ is generically a non-local functional.

We claim that $\varPhi_\uv := G_\uv\ast F_{\uv,1}[0]$ is a solution of Eq.~\eqref{eq:intro_mild}. The above statement is a consequence of the equalities $G_{\uv,\sIR}=G_\uv$ and $F_{\uv,\sIR}[\varphi]=F_\uv[\varphi]$ that hold for all $\sIR\in[0,\uv/2]$ and the identity
\begin{equation}\label{eq:intro_stationary_relation}
 F_{\uv,1}[0] = F_{\uv,\sIR}[G_{\uv,\sIR}\ast F_{\uv,1}[0]]
\end{equation}
that holds for all $\sIR\in[0,1]$. In order to prove the last identity we use the flow equation to show that the difference between the LHS and RHS of Eq.~\eqref{eq:intro_stationary_relation}, denoted by $g_{\uv,\sIR}$, satisfies the linear ODE
\begin{equation}
 \partial_\sIR g_{\uv,\sIR}
 =
 -\rD F_{\uv,\sIR}[G_{\uv,\sIR}\ast F_{\uv,1}[0],\partial_\sIR G_{\uv,\sIR}\ast g_{\uv,\sIR}]
\end{equation}
with the boundary condition $g_{\uv,1}=0$. This implies that $g_{\uv,\sIR}=0$ for all $\sIR\in[0,1]$.

\begin{rem}
The effective force plays also a central role in the approach to singular SPDEs proposed earlier by Kupiainen~\cite{kupiainen2016rg,kupiainen2017kpz} and applied to the dynamical $\varPhi^4_3$ model and the KPZ equation. The method developed by Kupiainen is based on the Wilsonian discrete renormalization group theory~\cite{wilson1971}. In this approach one uses the fact that for $\sIR\geq\uIR$ the effective force satisfies the equation
\begin{equation}
 F_{\uv,\sIR}[\varphi] = F_{\uv,\uIR}[(G_{\uv,\uIR}-G_{\uv,\sIR})\ast F_{\uv,\sIR}[\varphi]+\varphi].
\end{equation}
Given $F_{\uv,\uIR}[\varphi]$ the above equation is viewed as an equation for $F_{\uv,\sIR}[\varphi]$ that can be solved using the Banach fixed point theorem. One defines recursively the effective force $F_{\uv,\sIR}[\varphi]$ for $\uv=2L^{-N}$ and $\sIR\in\{L^{-N},\ldots,1\}$, where $L>1$, starting with $F_{2L^{-N}}[\varphi]=F_{2L^{-N},L^{-N}}[\varphi]$ and finishing with $F_{2L^{-N},1}[\varphi]$. In order to prove uniform bounds for $F_{2L^{-N},1}[\varphi]$ one has to appropriately adjust the counterterms which are the coefficients of the force $F_{2L^{-N}}[\varphi]$. The fine-tuning problem becomes exceedingly difficult for equations close to criticality. The flow equation provides a different, more efficient and simpler method of constructing the effective force that allows to treat equations arbitrarily close to criticality. 
\end{rem}

\section{Construction of the effective force coefficients}

The starting point of the construction of the effective force is the formal ansatz
\begin{equation}\label{eq:intro_ansatz}
 \langle F_{\uv,\sIR}[\varphi],\psi\rangle
 :=\sum_{i=0}^\infty \sum_{m=0}^\infty \lambda^i\,\langle F^{i,m}_{\uv,\sIR},\psi\otimes\varphi^{\otimes m}\rangle,
\end{equation}
where $\lambda$ is the prefactor of the cubic non-linearity in the original equation. The distributions $F^{i,m}_{\uv,\sIR}\in\sS'(\bM^{1+m})$, $i,m\in\bN_0$, are called the effective force coefficients. By definition the expression $\langle F^{i,m}_{\uv,\sIR},\psi\otimes\varphi_1\otimes\ldots\otimes\varphi_m\rangle$ is invariant under permutations of the test functions $\varphi_1,\ldots,\varphi_m\in\ \sS(\bM)$. The coefficients $F^{i,m}_\uv$ of the force $F_\uv[\varphi]$ are defined by an equality analogous to Eq.~\eqref{eq:intro_ansatz}.

\begin{rem}
Let us list the non-vanishing force coefficients $F^{i,m}_{\uv}$:
\begin{gather}
 F^{0,0}_\uv(x)=\xi(x),\qquad 
 F^{1,3}_\uv(x;x_1,x_2,x_3)=\delta_\bM(x-x_1)\delta_\bM(x-x_2)\delta_\bM(x-x_3),
 \\
 F^{i,1}_\uv(x;x_1)=c_{\uv}^{[i]} \delta_\bM(x-x_1),
 \quad i\in\{0,\ldots,i_\sharp\},
\end{gather}
where $\delta_\bM\in\sS'(\bM)$ is the Dirac delta at $0\in\bM$.
\end{rem}

The flow equation~\eqref{eq:intro_flow_eq} for the effective force $F_{\uv,\sIR}[\varphi]$ formally implies that the effective force coefficients $F^{i,m}_{\uv,\sIR}$ satisfy the flow equation
\begin{multline}\label{eq:intro_flow_eq_i_m}
 \langle \partial_\sIR^{\phantom{i}} F^{i,m}_{\uv,\sIR}\,,
 \,\psi\otimes\varphi^{\otimes m}\rangle
 \\
 =
 -\sum_{j=0}^i\sum_{k=0}^m
 \,(1+k)
 \,\big\langle F^{j,1+k}_{\uv,\sIR}\otimes F^{i-j,m-k}_{\uv,\sIR},
 \psi\otimes \varphi^{\otimes k}
 \otimes \fV\partial_\sIR G_{\uv,\sIR}
 \otimes \varphi^{\otimes(m-k)}
 \big\rangle
\end{multline}
with the boundary condition $F^{i,m}_{\uv,0}=F^{i,m}_{\uv}$, where $\fV\partial_\sIR G_{\uv,\sIR}\in C^\infty(\bM\times\bM)$ is defined by $\fV\partial_\sIR G_{\uv,\sIR}(x,y):=\partial_\sIR G_{\uv,\sIR}(x-y)$.

The basic idea behind the flow equation approach is a recursive construction of the effective force coefficients $F^{i,m}_{\uv,\sIR}$:
\begin{enumerate}
 \item[(0)] We set $F^{0,0}_{\uv,\sIR}=\xi$ and $F^{i,m}_{\uv,\sIR}= 0$ if $m>3i$.
 \item[(I)] Assuming that all $F^{i,m}_{\uv,\sIR}$ with $i<i_\circ$, or $i=i_\circ$ and $m>m_\circ$ were constructed we define $\partial_\sIR F^{i,m}_{\uv,\sIR}$ with $i=i_\circ$ and $m=m_\circ$ with the use of Eq.~\eqref{eq:intro_flow_eq_i_m}.
 \item[(II)] Subsequently, $F^{i,m}_{\uv,\sIR}$ is defined by $F^{i,m}_{\uv,\sIR} = F^{i,m}_{\uv} + \int_0^\sIR \partial_\uIR F^{i,m}_{\uv,\uIR}\,\rd\uIR$.
\end{enumerate}
Using this procedure we construct all the effective force coefficients $F^{i,m}_{\uv,\sIR}$ for arbitrary $\uv\in(0,1/2]$, $\sIR\in[0,1]$. 

\begin{rem}\label{rem:explicit_coefficients}
One easily shows that the effective force coefficients $F^{i,m}_{\uv,\sIR}$ actually vanish if $i=0$ and $m>0$, or $i>0$ and $m>2(i-1)+3$. The only non-zero coefficients $F^{i,m}_{\uv,\sIR}$ which are independent of the value of the flow parameter $\sIR$ are $F^{0,0}_{\uv,\sIR}(x)=\xi(x)$ and $F^{1,3}_{\uv,\sIR}(x;x_1,x_2,x_3)=F^{1,3}_{\uv}(x;x_1,x_2,x_3)$. These coefficients happen to be independent of the UV cutoff $\uv$. Let us give some further examples:
\begin{equation}
\begin{gathered}
 F^{1,2}_{\uv,\sIR}(x;x_1,x_2)=3\varPsi_{\uv,\sIR}(x)\, \delta_\bM(x-x_1)\delta_\bM(x-x_2),
 \\
 F^{1,1}_{\uv,\sIR}(x;x_1)=(3\varPsi_{\uv,\sIR}^2(x)+c_{\uv}^{[1]})\, \delta_\bM(x-x_1),
 \qquad
 F^{1,0}_{\uv,\sIR}(x)=\varPsi_{\uv,\sIR}^3(x)+c_{\uv}^{[1]}\,\varPsi_{\uv,\sIR}(x),
\end{gathered} 
\end{equation}
where $\varPsi_{\uv,\sIR}:=G_{\uv\shortparallel\sIR}\ast \xi$ and $G_{\uv\shortparallel\sIR}:=G_\uv-G_{\uv,\sIR}$  is the so-called fluctuation propagator. The coefficient $F^{2,5}_{\uv,\sIR}(x;x_1,\ldots,x_5)$ is obtained from the distribution
$3\,\delta_\bM(x-x_1)\delta_\bM(x-x_2)\,G_{\uv\shortparallel\sIR}(x-x_3)\, \delta_\bM(x_3-x_4)\delta_\bM(x_3-x_5)$ by symmetrization in variables $x_1,\ldots,x_5$. We also have
\begin{multline}
 F^{2,1}_{\uv,\sIR}(x;x_1) = 
 (3\varPsi_{\uv,\sIR}^2(x)+c_{\uv}^{[1]})\,
 G_{\uv\shortparallel\sIR}(x-x_1)\,
 (3\varPsi_{\uv,\sIR}^2(x_1)+c_{\uv}^{[1]})
 \\
 +\left(\int_{\bM}6\varPsi_{\uv,\sIR}(x)\,G_{\uv\shortparallel\sIR}(x-x_2)\,
 (\varPsi_{\uv,\sIR}^3(x_2)+c_{\uv}^{[1]}\varPsi_{\uv,\sIR}(x_2))\,\rd x_2
 + c_{\uv}^{[2]}\right)\delta_\bM(x-x_1),
\end{multline}
\begin{equation}
 F^{2,0}_{\uv,\sIR}(x) = 
 (3\varPsi_{\uv,\sIR}^2(x)+c_{\uv}^{[1]})
 \!\int\! G_{\uv\shortparallel\sIR}(x-x_1)
 (\varPsi_{\uv,\sIR}^3(x_1)+c_{\uv}^{[1]}\varPsi_{\uv,\sIR}(x_1))\,\rd x_1
 + c_{\uv}^{[2]}\, \varPsi_{\uv,\sIR}(x).
\end{equation}
For the sake of brevity, we did not give expressions for the coefficients $F^{2,4}_{\uv,\sIR}$, $F^{2,3}_{\uv,\sIR}$, $F^{2,2}_{\uv,\sIR}$, which should be constructed after $F^{2,5}_{\uv,\sIR}$ and before $F^{2,1}_{\uv,\sIR}$.
\end{rem}

\begin{rem}
The effective force is an analog of the effective potential in QFT. A recursive construction of the effective potential coefficients based on the flow equation is the backbone of a very simple proof of perturbative renormalizability of QFT models proposed by Polchinski~\cite{polchinski1984} (see~\cite{muller2003} for a review).
\end{rem}

The solution of Eq.~\eqref{eq:intro_mild} is formally given by the sum
\begin{equation}\label{eq:solution_series}
 \varPhi_\uv = G_\uv\ast F_{\uv,1}[0] := G_\uv\ast \sum_{i=0}^\infty \lambda^i F^{i,0}_{\uv,1}.
\end{equation}
Assuming that $\lambda$ is sufficiently small in Sec.~\ref{sec:solution} we prove that the above series converges absolutely and $\varPhi_\uv$, as defined above, solves Eq.~\eqref{eq:intro_mild} and converges almost surely as $\uv\searrow0$ in the Besov space $\mathscr{C}^\beta(\bM)$ for every $\beta<\sigma-\rdim/2$. To this end, we will establish certain bounds for $\partial_\uv^r\partial_\sIR^s F_{\uv,\sIR}^{i,m}$ which are stated in Sec.~\ref{sec:deterministic}. The bounds involve a regularizing kernel $K_\sIR$, which is introduced in Sec.~\ref{sec:kernels}, and a norm $\|\Cdot\|_{\cV^m}$, which is introduced in Sec.~\ref{sec:topology}.

\section{Regularizing kernels}\label{sec:kernels}

\begin{dfn}
For $n\in\bN_+$ let $\cK^n\subset\sS'(\bM^n)$ be the space of signed measures $K$ on $\bM^n$ with finite total variation $|K|$. We set $\|K\|_{\cK^n} = \int_{\bM^n} |K(\rd x_1\ldots\rd x_n)|$. If $n=1$, then we write \mbox{$\cK^1=\cK\subset\sS'(\bM)$}. We denote by $\delta_\bM\in\cK$ the Dirac delta at $0\in\bM$. Given $K\in\cK$ and $n\in\bN_+$ we set $K^{\otimes n}:=K\otimes\ldots\otimes K\in\cK^n$.
\end{dfn}

\begin{rem} It holds that $\|K\|_{\cK^n}=\|K\|_{L^1(\bM^n)}$ for all $K\in L^1(\bM^n)\subset\cK^n$.
\end{rem}

\begin{dfn}
For $\sIR\geq 0$ the kernel $K_\sIR\in\cK$ is the unique solution of the equation $\fP_\sIR K_\sIR=\delta_\bM$, where $\fP_\sIR:=1-[\sIR]^2\Delta$ and $[\sIR]:=\sIR^{1/\sigma}$. We set $K_\sIR^{\ast0}:=\delta_\bM$, $\fP^0_\sIR:=1$ and $K^{\ast(\oo+1)}_\sIR:=K_\sIR\ast K^{\ast\oo}_\sIR$, $\fP^{\oo+1}_\sIR:=\fP_\sIR\, \fP^{\oo}_\sIR$ for $\oo\in\bN_0$. We omit the index $\sIR$ if $\sIR=1$. 
\end{dfn}

\begin{rem}
It holds that $K_0=\delta_{\bM}$ and $K_\sIR\in L^1(\bM)$ for $\sIR>0$. For $\oo\in\bN_0$ and $\sIR\geq 0$ the kernel $K_\sIR^{\ast\oo}$ is a positive measure with total mass $\|K_\sIR^{\ast\oo}\|_\cK=1$. We have $\fP^\oo_\sIR K^{\ast\oo}_\sIR=\delta_\bM$. The fact that the regularizing kernel $K_\sIR^{\ast\oo}$ is an inverse of a differential operator simplifies the analysis in Sec.~\ref{sec:taylor}.
\end{rem}

\begin{lem}\label{lem:kernel_u_v}
For any $\sIR,\uIR>0$ it holds that
\begin{equation}\label{eq:kernel_u_v}
\begin{gathered}
 K_\sIR =  \fP_\uIR K_\sIR \ast K_\uIR,
 \qquad
 \|\fP_\uIR K_\sIR\|_\cK= 1\vee(2[\uIR/\sIR]^2-1).
\end{gathered} 
\end{equation}
In particular, if $\sIR\geq\uIR$, then $\|\fP_\uIR K_\sIR\|_\cK=1$.
\end{lem}
\begin{proof}
We have $\fP_\uIR K_\sIR = [\uIR/\sIR]^2\delta_\bM + (1-[\uIR/\sIR]^2) K_\sIR$ and $\|\delta_\bM\|_\cK=\|K_\sIR\|_{\cK}=1$.
\end{proof}

\begin{dfn}\label{dfn:periodization}
Let $\bT=\bM/(2\pi\bZ)^\rdim$. For $V\in L^1(\bM)$ we define $\fT V\in L^1(\bT)$ by
\begin{equation}
 \fT V(x):= \sum_{y\in(2\pi\bZ)^\rdim} V(x+y).
\end{equation}
\end{dfn}
\begin{rem}\label{rem:periodic}
For $K\in L^1(\bM)$ and periodic $f\in C(\bT)$ it holds that $K\ast f = \fT K \star f$, where $\ast$ and $\star$ are the convolutions in $\bM$ and $\bT$, respectively.
\end{rem}

\begin{lem}\label{lem:kernel_simple_fact}
Let $\oo\in\bN_0$, $a\in\bN_0^{\rdim}$ and $n\in[1,\infty]$. The following is true:
\begin{enumerate}
\item[(A)]
If $|a|\leq\oo$, then $\|\partial^a K^{\ast\oo}_\sIR\|_\cK
 \lesssim [\sIR]^{-|a|}$ uniformly in $\sIR>0$. 
\item[(B)] 
$\|\fT K^{\ast\rdim}_\sIR\|_{L^n(\bT)}\lesssim [\sIR]^{-\rdim (n-1)/n}$ uniformly in $\sIR\in(0,1]$.
\item[(C)]
It holds that $\|\fP_\sIR\partial_\sIR K_\sIR\|_{\cK} \lesssim [\sIR]^{-\sigma}$ uniformly in $\sIR>0$.
\end{enumerate}
\end{lem}

\begin{rem}
Let $f\in C(\bM)$, $\oo\in\bN_0$ and $a\in\bN_0^\rdim$ be such that $|a|\leq\oo$. Then $K^{\ast\oo}_{\sIR}\ast f\in C^\oo(\bM)$ and $\|\partial^a K^{\ast\oo}_{\sIR}\ast f\|_{L^\infty(\bM)}\lesssim [\sIR]^{-|a|}\|f\|_{L^\infty(\bM)}$ uniformly over $\sIR>0$ and $f\in C(\bM)$.
\end{rem}

\section{Function spaces for the coefficients}\label{sec:topology}

\begin{dfn}
For $\alpha<0$ and $\phi\in \sS'(\bM)$ we define
\begin{equation}
 \|\phi\|_{\sC^\alpha(\bM)}:= \sup_{\sIR\in(0,1]} [\sIR]^{-\alpha}\,
 \|K_\sIR^{\ast\oo}\ast \phi\|_{L^\infty(\bM)},
 \qquad
 \oo=\ceil{-\alpha}\in\bN_+.
\end{equation}
The space $\sC^\alpha(\bM)$ consists of $\phi\in\sS'(\bM)$ such that $\|\phi\|_{\sC^\alpha(\bM)}<\infty$. 
\end{dfn}

\begin{dfn}\label{dfn:sV}
Let $m\in\bN_0$. The vector space $\cV^m$ consists of $V\in C(\bM^{1+m})$ such that
\begin{equation}
 \|V\|_{\cV^m}
 :=
 \\
 \sup_{x\in\bM} \int_{\bM^m}
 |V(x;y_1,\ldots,y_m)|\,\rd y_1\ldots\rd y_m
\end{equation}
is finite and the function $x\mapsto V(x;y_1+x,\ldots,y_m+x)$ is $2\pi$ periodic for every $y_1,\ldots,y_m\in\bM$. For $\oo\in\bN_0$ the space $\cD^{m;\oo}$ consists of $V\in \sS'(\bM^{1+m})$ such that it holds $K^{\ast\oo,\otimes (1+m)}\ast V\in\cV^{m}$. The space $\cD^{m}$ is the union of the spaces $\cD^{m;\oo}$, $\oo\in\bN_0$. We also set $\cV^0=\cV$ and $\cD^0=\cD$.
\end{dfn}

\begin{rem}\label{rem:Vm_K}
For $V\in\cV^m$ and $K\in\cK^{1+m}$ it holds $\|K\ast V\|_{\cV^m}\leq \|K\|_{\cK^{1+m}} \|V\|_{\cV^m}$.
\end{rem}

\begin{rem}
It holds that $\cV=\cD^{0;0}=C(\bT)$. Moreover, we have $\|v\|_{\cV}=\|v\|_{L^\infty(\bM)}$ for all $v\in\cV$. Furthermore, $\cD=\sS'(\bT)$, since $K$ is the inverse of a differential operator.
\end{rem}

\begin{dfn}\label{dfn:map_Y}
The permutation group of the set $\{1,\ldots,n\}$ is denoted by $\cP_n$. For $m\in\bN_0$ and $V\in\cD^m$ and $\pi\in\cP_m$ we define $\fY_\pi V\in\cD^m$ by
\begin{equation}
 \big\langle \fY_\pi V,\psi \otimes
 \botimes_{q=1}^m \varphi_q\big\rangle
 :=
 \big\langle V,\psi\otimes
 \botimes_{q=1}^{m} \varphi_{\pi(q)}\rangle,
\end{equation}
where $\psi,\varphi_1,\ldots,\varphi_m\in\ \sS(\bM)$.
\end{dfn}
\begin{rem}\label{rem:fY_pi}
The map $\fY_\pi:\cV^{m}\to\cV^{m}$ is well defined and has norm one. 
\end{rem}

\begin{dfn}\label{dfn:map_B}
Let $m\in\bN_0$, $k\in\{0,\ldots,m\}$. We define the trilinear map $\fB\,:\,\sS(\bM)\times \cV^{1+k}\times\cV^{m-k}\to\cV^m$ by
\begin{multline}\label{eq:fB1_dfn}
 \fB(G,W,U)(x;y_1,\ldots,y_m)
 \\
 :=
 \int_{\bM^2} W(x;y_0,\ldots,y_k)
 \,G(y_0-z)\, 
 U(z;y_{1+k},\ldots,y_m)\,\rd y_0\rd z.
\end{multline}
\end{dfn}

\begin{lem}\label{lem:fB1_bound}
The map $\fB:\sS(\bM)\times \cV^{1+k}\times\cV^{m-k}\to\cV^m$ is well defined and
\begin{equation}\label{eq:fB1_bound}
 \|\fB(G,W,U)\|_{\cV^{m}}
 \leq
 \|G\|_\cK\,\|W\|_{\cV^{1+k}} \|U\|_{\cV^{m-k}}.
\end{equation}
\end{lem}
\begin{proof}
To prove well-definedness note that
\begin{multline}
 \fB(G,W,U)(x;y_1+x,\ldots,y_m+x)
 \\
 =
 \int_{\bM^2} W(x;y_0+x,\ldots,y_k+x)
 \,G(y_0-z)\, 
 U(z+x;y_{1+k}+x,\ldots,y_m+x)\,\rd y_0\rd z
\end{multline}
is $2\pi$ periodic. We have
\begin{multline}
 \|\fB(G,W,U)\|_{\cV^{m}}=
 \sup_{x\in\bM}  \int_{\bM^{m}} |\fB(G,W,U)(x;y_1,\ldots,y_m)|\,\rd y_1\ldots\rd y_m
 \\
 \leq
 \sup_{x\in\bM}\int_{\bM^{2+m}} |W(x;y_0,\ldots,y_k)|\,
  |G(y_0-z)|\, |U(z;y_{1+k},\ldots,y_{m})|\, \rd z\rd y_0\ldots\rd y_m.
\end{multline}
It is easy to see that the last line is bounded by the RHS of~\eqref{eq:fB1_bound}.
\end{proof}

\begin{rem}\label{rem:fB_Ks}
The fact that $\fP_\sIR^\oo K_\sIR^{\ast\oo}=\delta_\bM$ implies that for all $\sIR>0$ it holds that
\begin{equation}
\begin{gathered}
 K_\sIR^{\ast\oo,\otimes(1+m)}\ast\fB(G,W,U)=
 \fB\big(\fP^{2\oo}_\sIR G,
 K_\sIR^{\ast\oo,\otimes(2+k)}\ast W,
 K_\sIR^{\ast\oo,\otimes(1+m-k)}\ast U\big).
\end{gathered} 
\end{equation}
This allows to define $\fB(G,W,U)\in\cD^m$ for all $G\in\sS(\bM)$, $W\in\cD^{1+k}$ and $U\in\cD^{m-k}$.
\end{rem}

\begin{rem}
The RHS of Eq.~\eqref{eq:intro_flow_eq_i_m} can be written compactly using the map $\fB$.
\end{rem}

\section{Deterministic bounds for the irrelevant coefficients}\label{sec:deterministic}

\begin{dfn}
Recall that $\sigma\in(\rdim/3,\rdim/2]$, $[\sIR]=\sIR^{1/\sigma}$ and set
\begin{equation}
 \dim(\xi):=\rdim/2>0,\qquad 
 \dim(\varPhi):=\rdim/2-\sigma \geq 0,\qquad
 \dim(\lambda):=3\sigma-\rdim>0.
\end{equation} 
\end{dfn}

\begin{dfn}\label{dfn:relevant_irrelevant}
For $\varepsilon\geq0$ and $i,m\in\bN_0$ we define
\begin{equation}
 \varrho_\varepsilon(i,m) := 
 -\dim(\xi)-\varepsilon
 + m\, (\dim(\Phi)+3\varepsilon)
 + i\, (\dim(\lambda)-9\varepsilon)\in \bR.
\end{equation}
We omit $\varepsilon$ if $\varepsilon=0$. The effective force coefficients $F^{i,m}_{\uv,\sIR}$ such that $\varrho(i,m)\leq 0$ are called relevant. The remaining coefficients are called irrelevant.
\end{dfn}
\begin{rem}
The number of relevant coefficients $F^{i,m}_{\uv,\sIR}$ such that $m\leq3i$ is always finite (recall that $F^{i,m}_{\uv,\sIR}=0$ if $m>3i$). For example, in the case $\rdim=5$ and $\sigma=2$ the relevant coefficients are $F^{0,0}_{\uv,\sIR}$, $F^{1,0}_{\uv,\sIR}$, $F^{1,1}_{\uv,\sIR}$, $F^{1,2}_{\uv,\sIR}$, $F^{1,3}_{\uv,\sIR}$, $F^{2,0}_{\uv,\sIR}$, $F^{2,1}_{\uv,\sIR}$. For explicit expressions for these coefficients see Remark~\ref{rem:explicit_coefficients}.
\end{rem}

\begin{rem}\label{rem:rho}
For arbitrary $\varepsilon>0$ and $i,m\in\bN_0$ such that $m\leq 3i$, $\varrho_\varepsilon(i,m)<\varrho(i,m)$ holds. Let $i_\diamond$ be the smallest integer such that $\varrho(i_\diamond+1,0)>0$. Then $i_\diamond\in\bN_+$. Moreover, let $\varrho_\diamond>0$ be the minimum of $\varrho(i,m)+\ro$ for $i\in\{0,\ldots,i_\diamond\}$, $m\in\{0,\ldots,3i\}$, $\ro\in\bN_+$ such that $\varrho(i,m)+\ro>0$. Define $\varepsilon_\diamond:=\dim(\xi)/3\wedge\dim(\lambda)/11\wedge\varrho_\diamond/(10+9i_\diamond)\wedge\sigma$. We claim that for all $\varepsilon\in(0,\varepsilon_\diamond)$ and all \mbox{$i,m,\ro\in\bN_0$} it holds that \mbox{$\varrho_\varepsilon(i,m)+\ro>0$} if \mbox{$\varrho(i,m)+\ro>0$}. In what follows, we fix some $\varepsilon\in(0,\varepsilon_\diamond/3)$.
\end{rem}

\begin{lem}\label{lem:kernel_G}
For all $\oo\in\bN_0$, $r\in\bN_0$ it holds that $\partial_\uv^r \partial_\sIR G_{\uv,\sIR}\in C^\infty_\rc(\bM)$ and
\begin{equation}\label{eq:bound_G}
 \|\fP^{\oo}_\sIR \partial_\uv^r \partial_\sIR G_{\uv,\sIR}\|_\cK \lesssim [\uv]^{(\varepsilon-\sigma)r}\,[\sIR]^{-\varepsilon r}
\end{equation} 
uniformly in $\uv\in(0,1/2]$, $\sIR\in(0,1]$, where $G_{\uv,\sIR}$ was introduced in Def.~\ref{dfn:kernel_G}.
\end{lem}
\begin{proof}
Let $a\in\bN_0^\rdim$. It holds that $|\partial^a G(x)|\lesssim |x|^{\sigma-\rdim-|a|}$ uniformly for $|x|\leq 2$. If $\sigma\in2\bN_+$, then $\partial^a G$ decays fast at infinity. In general, $|\partial^a G(x)|\lesssim |x|^{-\sigma-\rdim-|a|}$ uniformly for $|x|> 1$. Moreover, we have
\begin{equation}
 |\partial^a \partial_\uv^r \chi_\uv(|x|^\sigma)|\lesssim [\uv]^{(\varepsilon-\sigma)r} |x|^{-\varepsilon r-|a|},
 \qquad
 |\partial^a \partial_\sIR\chi_\sIR(|x|^\sigma)|\lesssim |x|^{\sigma-|a|}/\sIR^2
\end{equation}
uniformly in $\uv\in(0,1/2]$, $\sIR\in(0,1]$ and $x\in\bM$. Furthermore, $|\partial^a \partial_\sIR\chi_\sIR(|x|^\sigma)|$ vanishes unless $\sIR<(1-\sIR)|x|^\sigma\leq 2\sIR$. Using the above properties and considering separately $\sIR\in(0,1/2]$ and $\sIR\in(1/2,1]$ we obtain $\|\partial^a \partial_\uv^r \partial_\sIR G_{\uv,\sIR}\|_\cK \lesssim [\uv]^{(\varepsilon-\sigma)r}[\sIR]^{-\varepsilon r-|a|}$. This implies the lemma since $\fP_\sIR^\oo = (1-[\sIR]^2\Delta_x)^\oo$.
\end{proof}

\begin{rem}\label{rem:tilde_R}
Given $\oo\in\bN_0$ there exists $\tilde R>0$ such that 
\begin{equation}
 \frac{(1+m)\,\sigma}{\varrho_{3\varepsilon}(i,m)}\leq \tilde R^{1/2},
 \qquad
 \|\fP^{2\oo}_\sIR \partial_\uv^r \partial_\sIR G_{\uv,\sIR}\|_\cK \leq 1/6\,\tilde R^{1/2}\, [\uv]^{(\varepsilon-\sigma)r}\,[\sIR]^{-\varepsilon}
\end{equation}
for all $i,m\in\bN_0$ such that $\varrho(i,m)>0$ and all $r\in\{0,1\}$, $\uv\in(0,1/2]$, $\sIR\in(0,1]$. The first of the above bounds implies, in particular, that $\sigma/(\dim(\varPhi)+9\varepsilon)\leq\tilde R^{1/2}$.
\end{rem}

\begin{thm}\label{thm:deterministic_convergence}
Fix $\oo\in\bN_0$. Let $\tilde R>1$ as in Remark~\ref{rem:tilde_R}. Assume that for $r\in\{0,1\}$, $s=0$, all $i,m\in\bN_0$ such that $\varrho(i,m)\leq0$ and all $\uv\in(0,1/2]$, $\sIR\in(0,1]$ the following bound holds 
\begin{equation}\label{eq:bound_deterministic_main}
 \|K_{\sIR}^{\ast\oo,\otimes(1+m)}\ast \partial_\uv^r\partial_\sIR^s F_{\uv,\sIR}^{i,m}\|_{\cV^m}
 \!\leq\! \frac{\tilde R^{1-s/2+2(3i-m)}}{4(1+i)^2\,4(1+m)^{2-s}} [\uv]^{(\varepsilon-\sigma)r}\, [\sIR]^{\varrho_{3\varepsilon}(i,m)-\sigma s}.
\end{equation}
Then the above bound holds for all $r,s\in\{0,1\}$, $i,m\in\bN_0$, $\uv\in(0,1/2]$, $\sIR\in(0,1]$.
\end{thm}
\begin{rem}
By Theorems~\ref{thm:cumulants},~\ref{thm:probabilistic_bounds},~\ref{thm:deterministic_taylor},~\ref{thm:deterministic_long_range} (applied in this order), there exists a choice of counterterms such that the assumption of the above theorem is satisfied for some random $\tilde R>0$ such that $\bE \tilde R^n <\infty$ for all $n\in\bN_+$.
\end{rem}

\begin{rem}
If the bound~\eqref{eq:bound_deterministic_main} holds for $\oo=\ooo$, then it also holds for all $\oo>\ooo$.
\end{rem}

\begin{rem}
Let us comment on the assumption in two simple cases. For $i=0$, $m=0$ the bound~\eqref{eq:bound_deterministic_main} says that \mbox{$\|K_\sIR^{\ast\oo}\ast\xi\|_{\cV}\leq \tilde R/4^2\,[\sIR]^{-\dim(\xi)-3\varepsilon}$}, which is known to be true for $\oo=\rdim$ and some random $\tilde R>0$ such that $\bE \tilde R^n<\infty$ for all $n\in\bN_+$. For $i=1$, $m=3$ the bound~\eqref{eq:bound_deterministic_main} says that $\|K_\sIR^{\ast\oo,\otimes4}\ast F^{1,3}_{\uv,\sIR}\|_{\cV^3}\leq \tilde R/4^5\,[\sIR]^{-\varepsilon}$. It is easy to see that the above bound is satisfied for $\tilde R=4^5$ using the fact that $F^{1,3}_{\uv,\sIR}(x;x_1,x_2,x_3)=\delta_\bM(x-x_1)\delta_\bM(x-x_2)\delta_\bM(x-x_3)$.
\end{rem}
\begin{rem}
For $\lambda\leq \tilde R^{-6}$ the bounds~\eqref{eq:bound_deterministic_main} imply convergence of the series~\eqref{eq:solution_series} defining $\varPhi_\uv$, which is our candidate for the solution of Eq.~\eqref{eq:intro_mild}. In fact, one easily proves that $\|K^{\ast\oo}\ast\partial_\uv^r \varPhi_\uv\|_{\cV}\leq R\,[\uv]^{(\varepsilon-\sigma)r}$ for $r\in\{0,1\}$ and all $\uv\in(0,1]$. This implies the existence of the limit $\lim_{\uv\searrow0}\varPhi_\uv$ in $\sS'(\bM)$. In Sec.~\ref{sec:solution} we will prove that $\varPhi_\uv$ solves Eq.~\eqref{eq:solution_series} and $\lim_{\uv\searrow0}\varPhi_\uv$ exists in $\sC^{-\dim(\varPhi)-9\varepsilon}(\bM)$.
\end{rem}

\begin{proof}
Fix some $i_\circ\in\bN_+$, $m_\circ\in\bN_0$ and assume that the statement with $s=0$ is true for all $i,m\in\bN_+$ such that either $i<i_\circ$ or $i=i_\circ$ and $m>m_\circ$. We shall prove the statement for $i=i_\circ$, $m=m_\circ$. We first consider the case $s=1$. Using the flow equation~\eqref{eq:intro_flow_eq_i_m} and the notation introduced in Sec.~\ref{sec:topology} we obtain
\begin{multline}\label{eq:flow_deterministic_i_m}
 \partial_\sIR \partial_\uv^r F^{i,m}_{\uv,\sIR}
 =
 -\frac{1}{m!}
 \sum_{\pi\in\cP_m}\sum_{j=0}^i\sum_{k=0}^m\sum_{u+v+w=r} 
 (1+k)\,
 \\
 \times \fY_\pi\fB\big(\partial_\uv^u\partial_\sIR G_{\uv,\sIR},\partial_\uv^v F^{j,1+k}_{\uv,\sIR},\partial_\uv^wF^{i-j,m-k}_{\uv,\sIR}\big),
\end{multline}
where $r,u,v,w\in\{0,1\}$. By Remark~\ref{rem:fY_pi}, Lemma~\ref{lem:fB1_bound} and Remark~\ref{rem:fB_Ks} we get
\begin{multline}
 \|K_{\sIR}^{\ast\oo,\otimes(1+m)}\ast\partial_\sIR \partial_\uv^r F^{i,m}_{\uv,\sIR}\|_{\cV^m}
 \leq 
 (1+m)~\sum_{u+v+w=r} \|\fP_\sIR^\oo\partial_\uv^u\partial_\sIR G_{\uv,\sIR}\|_\cK 
 \\
 \times \sum_{j=0}^i\sum_{k=0}^m\|K_{\sIR}^{\ast\oo,\otimes(2+k)}\ast \partial_\uv^v F^{j,1+k}_{\uv,\sIR}\|_{\cV^{1+k}}~
 \|K_{\sIR}^{\ast\oo,\otimes(1+m-k)}\partial_\uv^wF^{i-j,m-k}_{\uv,\sIR}\|_{\cV^{m-k}}.
\end{multline}
The statement of the theorem with $s=1$ follows now from the induction hypothesis, Remark~\ref{rem:tilde_R}, 
the inequality
\begin{equation}
 \varrho_{3\varepsilon}(i,m)-\sigma \leq \varrho_{3\varepsilon}(j,1+k) + \varrho_{3\varepsilon}(i-j,m-k)-\varepsilon
\end{equation}
and the bound
\begin{equation}
 \sum_{j=0}^i\sum_{k=0}^m \frac{\tilde R^{1+2(3j-k-1)}}{4(1+j)^2\,4(2+k)^2}
 \frac{\tilde R^{1+2(3i-3j-m+k)}}{4(1+i-j)^2\,4(1+m-k)^2}
 \leq \frac{\tilde R^{2(3i-m)}}{4(1+i)^2\,4(1+m)^2},
\end{equation}
which is a consequence of the inequality
\begin{equation}
 \sum_{j=0}^i \frac{1}{4(1+j)^2\,4(1+i-j)^2} \leq \frac{1}{4(1+i)^2}.
\end{equation}
In order to prove the statement of the theorem for $s=0$ and $i,m\in\bN_0$ such that $\varrho(i,m)>0$ we use the identity
\begin{equation}\label{eq:determinisitc_intagrate}
 F^{i,m}_{\uv,\sIR} = F^{i,m}_{\uv} + \int_0^\sIR \partial_\uIR F^{i,m}_{\uv,\uIR}\,\rd\uIR.
\end{equation}
We first observe that $F^{i,m}_{\uv,0}=F^{i,m}_\uv=0$ if $\varrho(i,m)>0$. Next, we note that
\begin{equation}\label{eq:determinisitc_bound_u_v}
 \|K_{\sIR}^{\ast\oo,\otimes(1+m)}\ast
\partial_\uv^r\partial_\uIR F_{\uv,\uIR}^{i,m}\|_{\cV^m}\leq\|K_{\uIR}^{\ast\oo,\otimes(1+m)}\ast
\partial_\uv^r\partial_\uIR F_{\uv,\uIR}^{i,m}\|_{\cV^m}
\end{equation}
for $\uIR\leq\sIR$ by Lemma~\ref{lem:kernel_u_v} and Remark~\ref{rem:Vm_K}. The statement of the theorem with $s=0$ follows now from the statement with $s=1$ and the bounds
\begin{equation}
 \|K_{\sIR}^{\ast\oo,\otimes(1+m)}\ast \partial_\uv^r F_{\uv,\sIR}^{i,m}\|_{\cV^m}
 \\
 \leq
 \int_{0}^{\sIR}
 \|K_{\uIR}^{\ast\oo,\otimes(1+m)}\ast \partial_\uv^r\partial_\uIR F_{\uv,\uIR}^{i,m}\|_{\cV^m}\,\rd\uIR
\end{equation}
and
\begin{equation}
 \int_{0}^\sIR [\uIR]^{\varrho_{3\varepsilon}(i,m)-\sigma}\,\rd\uIR \leq  \frac{\sigma\,[\sIR]^{\varrho_{3\varepsilon}(i,m)}}{\varrho_{3\varepsilon}(i,m)}
 \leq \frac{\tilde R^{1/2}\,[\sIR]^{\varrho_{3\varepsilon}(i,m)}}{1+m}
\end{equation}
We stress that the first inequality in the last bound is valid only if $\varrho_{3\varepsilon}(i,m)> 0$, which holds provided $\varrho(i,m)>0$ by Remark~\ref{rem:rho}.
\end{proof}

\section{Generalized effective force coefficients}\label{sec:modified_coefficients}

In this section we introduce generalized coefficients $F^{i,m,a}_{\uv,\sIR}$ and $f^{i,m,a}_{\uv,\sIR}$. Assuming certain bounds for $f^{i,m,a}_{\uv,\sIR}$, which can be verified using probabilistic methods, in Sec.~\ref{sec:deterministic_relavant} and~\ref{sec:long_range} we prove bounds for $F^{i,m,a}_{\uv,\sIR}$, which in particular imply the bounds~\eqref{eq:bound_deterministic_main} for all $F^{i,m}_{\uv,\sIR}=F^{i,m,0}_{\uv,\sIR}$ such that $\varrho(i,m)\leq0$. 

\begin{dfn}\label{dfn:modified_coefficients}
We denote the set of multi-indices by $\frM=\bN_0^\rdim$. For $m\in\bN_+$ and $a=(a_1,\ldots,a_m)\in\frM^m$ we define $|a|:=|a_1|+\ldots+|a_m|$. We also set $\frM^0\equiv\{0\}$. For $a\in\frM$ we define $\cX^a\in C^\infty(\bM)$ by $\cX^a(x):=x^a$ and for $a\in\frM^m$ we define $\cX^{m,a}\in C^\infty(\bM^{1+m})$ by
\begin{equation}
 \cX^{m,a}(x;y_1,\ldots,y_m)
 := 
 (x-y_1)^{a_1} \ldots (x-y_m)^{a_m}.
\end{equation}
For $i,m\in\bN_0$, $a\in\frM^m$ we define $F^{i,m,a}_{\uv,\sIR}\in\sS'(\bM^{1+m})$ and $f^{i,m,a}_{\uv,\sIR}\in\sS'(\bM)$ by
\begin{equation}
 F^{i,m,a}_{\uv,\sIR}:= \cX^{m,a} F^{i,m}_{\uv,\sIR},
 \qquad
 \langle f^{i,m,a}_{\uv,\sIR},\psi\rangle:=\langle F^{i,m,a}_{\uv,\sIR},\psi\otimes 1_\bM^{\otimes m}\rangle,
\end{equation}
where $\psi\in\sS(\bM)$ and $1_\bM(x)=1$ for all $x\in\bM$. The force coefficients $F^{i,m,a}_\uv$ and $f^{i,m,a}_\uv$ are defined analogously. The effective force coefficients $F^{i,m,a}_{\uv,\sIR}$ or $f^{i,m,a}_{\uv,\sIR}$ such that $\varrho(i,m)+|a|\leq 0$ are called relevant. The remaining coefficients are called irrelevant. The finite collection of all relevant coefficients $f^{i,m,a}_{\uv,\sIR}$ such that $m\leq 3i$ is called the enhanced noise (recall that $f^{i,m,a}_{\uv,\sIR}=0$ if $m>3i$).
\end{dfn}

\begin{rem}
Let us list the non-zero force coefficients $f^{i,m,a}_{\uv\phantom{0}}=f^{i,m,a}_{\uv,0}\in\sS'(\bM)$:
\begin{equation}
 f^{0,0,0}_\uv = \xi(x),\quad
 f^{1,3,0}_\uv = 1,\quad
 f^{i,1,0}_\uv = c^{[i]}_\uv,
 \quad i\in\{1,\ldots,i_\sharp\}.
\end{equation}
\end{rem}

\begin{rem}\label{rem:taylor}
Given $i,m\in\bN_0$, $a\in\frM^m$ such that $\varrho(i,m)+|a|\in(-\ro,1-\ro]$ for some $\ro\in\bN_+$ the relevant coefficient $F^{i,m,a}_{\uv,\sIR}$ can be expressed in terms of irrelevant coefficients $F^{i,m,b}_{\uv,\sIR}$, $|b|=\ro$, and relevant coefficients $f^{i,m,b}_{\uv,\sIR}$, $|b|<\ro$. The above fact plays a crucial role in the proof of bounds for the relevant coefficients $F^{i,m,a}_{\uv,\sIR}$ given in Sec.~\ref{sec:deterministic_relavant}. In Sec.~\ref{sec:taylor} we show that it is a consequence of the Taylor theorem. The relevant coefficients $f^{i,m,a}_{\uv,\sIR}\in\sS'(\bM)$, which are elements of the enhanced noise, are bounded using probabilistic methods. On the other hand, the irrelevant coefficients $F^{i,m,a}_{\uv,\sIR}\in\sS'(\bM^{1+m})$ are bounded deterministically using the same strategy as in the proof of Theorem~\ref{thm:deterministic_convergence}. 
\end{rem}

\begin{rem}
In the case $\rdim=5$ and $\sigma=2$ the enhanced noise consists of $f^{0,0,0}_{\uv,\sIR}=\xi$, $f^{1,0,0}_{\uv,\sIR}$, $f^{1,1,0}_{\uv,\sIR}$, $f^{1,2,0}_{\uv,\sIR}$, $f^{1,3,0}_{\uv,\sIR}=1$, $f^{2,0,0}_{\uv,\sIR}$, $f^{2,1,0}_{\uv,\sIR}$. The following coefficients are relevant:
\begin{equation}
\begin{gathered}
 F^{0,0}_{\uv,\sIR}(x)=f^{0,0,0}_{\uv,\sIR}(x)=\xi(x),
 \\
 F^{1,0}_{\uv,\sIR}(x)=f^{1,0,0}_{\uv,\sIR}(x),
 \qquad
 F^{1,1}_{\uv,\sIR}(x;x_1)=f^{1,1,0}_{\uv,\sIR}(x)\delta_\bM(x-x_1),
 \\
 F^{1,2}_{\uv,\sIR}(x;x_1,x_2)=f^{1,2,0}_{\uv,\sIR}(x)\delta_\bM(x-x_1)\delta_\bM(x-x_2),
 \\
 F^{1,3}_{\uv,\sIR}(x;x_1,x_2,x_3)=f^{1,3,0}_{\uv,\sIR}(x)\delta_\bM(x-x_1)\delta_\bM(x-x_2)\delta_\bM(x-x_3),
 \\
 F^{2,0}_{\uv,\sIR}(x)=f^{2,0,0}_{\uv,\sIR}(x),
 \qquad
 F^{2,1}_{\uv,\sIR}(x;x_1).
\end{gathered}
\end{equation}
The coefficient $F^{2,1}_{\uv,\sIR}(x;x_1)$ is non-local in the sense that it is not proportional to $\delta_\bM(x-x_1)$. The coefficient $F^{2,1}_{\uv,\sIR}$ can be expressed in terms of the relevant coefficient $f^{2,1,0}_{\uv,\sIR}$ and the irrelevant coefficient $F^{2,1,a}_{\uv,\sIR}$ with $|a|=1$. Since the coefficients $F^{1,m}_{\uv,\sIR}$ are local the coefficients $F^{1,m,a}_{\uv,\sIR}$ vanish identically for $m\in\bN_+$ and $a\neq 0$. Explicit expressions for all of the above coefficients can be easily deduced from expressions given in Remark~\ref{rem:explicit_coefficients}.
\end{rem}

For a list of multi-indices $a=(a_1,\ldots,a_m)\in\frM^m$ and a permutation $\pi\in\cP_n$ we set $\pi(a):=(a_{\pi(1)},\ldots,a_{\pi(m)})$. Note that
\begin{equation}\label{eq:poly_perm}
 \cX^{m,a}(x;y_1,\ldots,y_m)
 \\
 =
 \cX^{m,\pi(a)}(x;y_{\pi(1)},\ldots,y_{\pi(m)}).
\end{equation}
We claim that for all $y_0,z\in\bM$ it holds that
\begin{multline}\label{eq:poly_binom}
 \cX^{m,a}(x;y_1,\ldots,y_m)
 =
 \sum_{b,c,d} \frac{a!}{b!c!d!}\,
 \cX^{1+k,(b_{1+k}+\ldots+b_m,b_1,\ldots,b_k)}(x;y_0,y_1,\ldots,y_k)
 \\
 \times\cX^{1,c_{1+k}+\ldots+c_m}(y_0;z)
 \,\cX^{m-k,(d_{1+k},\ldots,d_m)}(z;y_{1+k},\ldots,y_m),
\end{multline}
where the sum is over all $b,c,d\in\frM^m$ such that $b_p=a_p$, $c_p=0$, $d_p=0$ for $p\in\{1,\ldots,k\}$ and $b_p+c_p+d_p=a_p$ for $p\in\{1+k,\ldots,m\}$. Throughout the paper we use the following schematic notation for the sums of the above type
\begin{multline}\label{eq:poly_binom_notation}
 \cX^{m,a}(x;y_1,\ldots,y_m)
 =\sum_{b+c+d=a} \frac{a!}{b!c!d!}\,\cX^{1+k,b}(x;y_0,y_1,\ldots,y_k)
 \\
 \times
 \cX^{c}(y_0-z) \cX^{m-k,d}(z;y_{1+k},\ldots,y_m).
\end{multline}
Let us note that the formula~\eqref{eq:poly_binom} above was obtained by rewriting $(x-y_p)^{a_p}$, $p\in\{1+k,\ldots,m\}$, as
\begin{equation}
 (x-y_p)^{a_p} 
 =\sum_{\substack{b_p,c_p,d_p\\b_p+c_p+d_p=a_p}}\frac{a_p!}{b_p!c_p!d_p!}~(x-y_0)^{b_p}(y_0-z)^{c_p}(z-y_p)^{d_p}.
\end{equation}

Using the flow equation~\eqref{eq:intro_flow_eq_i_m} and the identity~\eqref{eq:poly_binom} we show that the effective force coefficients $F^{i,m,a}_{\uv,\sIR}$ satisfy the flow equation
\begin{multline}\label{eq:flow_deterministic_i_m_a}
 \partial_\sIR^{\phantom{i}}\partial_\uv^r F^{i,m,a}_{\uv,\sIR}
 =
 -\frac{1}{m!}
 \sum_{\pi\in\cP_m}\sum_{j=0}^i\sum_{k=0}^m\,(1+k)\sum_{u+v+w=r}\frac{r!}{u!v!w!}
 \sum_{b+c+d=\pi(a)}\frac{a!}{b!c!d!}
 \\
 \times\fY_\pi\fB\big(\cX^c\partial_\uv^u\partial_\sIR G_{\uv,\sIR},\partial_\uv^v F^{j,1+k,b}_{\uv,\sIR},\partial_\uv^w F^{i-j,m-k,d}_{\uv,\sIR}\big),
\end{multline}
where $r\in\bN_0$, $u,v,w\in\bN_0$ and the sum over multi-indices $b,c,d$ is restricted to the set specified below Eq.~\eqref{eq:poly_binom}. Note that $\frac{r!}{u!v!w!}=1$ above if $r=1$.

\section{Taylor polynomial and remainder}\label{sec:taylor}

\begin{dfn}
For $a\in\frM^m$, $V\in C^\infty(\bM^{1+m})$ we define $\partial^a V\in C^\infty(\bM^{1+m})$ by $\partial^a V(x;x_1,\ldots,x_m):=
 \partial^{a_1}_{x_1} \ldots\partial^{a_m}_{x_m}
 V(x;x_1,\ldots,x_m)$.
The above map extends in an obvious way to $\sS'(\bM^{1+m})\supset\cD^m$. 
\end{dfn}

\begin{dfn}\label{dfn:map_L_m}
For $v\in\cD$ we define $\fL^m v\in\cD^m$ by the equality
\begin{equation}
 \langle\,\fL^m v\,,\,\psi\otimes\varphi_1\otimes\ldots\otimes\varphi_m\,\rangle= \langle v\,,\,\psi\varphi_1\ldots\varphi_m\rangle,
\end{equation}
where $\psi,\varphi_1,\ldots,\varphi_m\in \sS(\bM)$ are arbitrary.
\end{dfn}

\begin{dfn}\label{dfn:map_Z}
For $\tau>0$ and $V\in C(\bM^{1+m})$ we define $\fZ_\tau V\in C(\bM^{1+m})$ by the equality
\begin{equation}
 \fZ_\tau V(x;y_1,\ldots,y_m)
 :=\tau^{-\rdim m}\,V(x;x+(y_1-x)/\tau,\ldots,x+(y_m-x)/\tau).
\end{equation}
The above map extends in an obvious way to $\sS'(\bM^{1+m})\supset\cD^m$. 
\end{dfn}

\begin{dfn}\label{dfn:map_I}
The linear map $\fI\,:\,\cV^m\to\cV$ is defined by
\begin{equation}
 \fI V(x):= \int_{\bM^m} V(x;y_1,\ldots,y_m)\,\rd y_1\ldots\rd y_m.
\end{equation}
The map $\fI$ is extended to $V\in\cD^{m;\oo}$ by the formula
\begin{equation}
 \langle \fI V,\psi\rangle := \langle (\delta_\bM\otimes K^{\ast\oo,\otimes m})\ast V,\psi\otimes 1^{\otimes m}_\bM\rangle,
\end{equation}
where $\psi\in \sS(\bM)$.
\end{dfn}
\begin{lem}\label{lem:map_I}
The map $\fI$ is well defined and has the following properties.
\begin{enumerate}
\item[(A)] $\fI\big( K_\sIR^{\ast\oo,\otimes(1+m)} \ast V\big)=K^{\ast\oo}_\sIR\ast\fI V$
for all $V\in\cV^m$, $\oo\in\bN_0$ and $\sIR>0$.

\item[(B)] $\|\fI V\|_\cV \leq \|V\|_{\cV^m}$ for $V\in\cV^m$.
\end{enumerate}
\end{lem}
\begin{proof}
The well-definedness and Part (A) follow from $\int_\bM K_\sIR(x)\,\rd x=1$. Part~(B) is a consequence of Def.~\ref{dfn:sV} of the norm~$\cV^m$.
\end{proof}

\begin{dfn}\label{dfn:map_X}
For $\ro\in\bN_+$, $a\in\frM^m$ such that $|a|<\ro$ and two collections of distributions: $v^b\in\cD$, $b\in\frM^m$, $|b|<\ro$, and $V^b\in\cD^m$, $b\in\frM^m$, $|b|=\ro$, the distribution $\fX_\ro^a(v^b,V^b)\in \cD^m$ is defined by the equality 
\begin{equation}
 \fX_\ro^a(v^b,V^b)
 :=
 \sum_{|a+b|<\ro} \frac{1}{b!}
 \partial^b \fL^m(v^{a+b})
 +
 \sum_{|a+b|=\ro} \frac{|b|}{b!}
 \int_0^1 (1-\tau)^{|b|-1}\,
 \partial^b \fZ_\tau(V^{a+b})\,\rd\tau,
\end{equation} 
where the sums above are over $b\in\frM^m$.
\end{dfn}

\begin{thm}\label{thm:taylor_bounds}
Let $\ro\in\bN_+$ and $\oo\in\bN_0$. There exists a constant $c>0$ such that the following statement is true. Let $V^b\in\cV^m$, $v^b\in\cV$ be as in Def.~\ref{dfn:map_X} and $\sIR\in(0,1]$. Assume that there exists a constant $C>0$ such that for $b\in\frM$
\begin{equation}\label{eq:taylor_ass_V}
 \|K_\sIR^{\ast\oo,\otimes(1+m)}\ast V^b\|_{\cV^m} \leq C\, [\sIR]^{|b|},\quad |b|=\ro,
\quad\quad
 \|K_\sIR^{\ast\oo}\ast v^b\|_\cV \leq C\, [\sIR]^{|b|}.
 \quad |b|<\ro,
\end{equation}
Then for $a\in\frM$ such that $|a|<\ro$, it holds that
\begin{equation}\label{eq:taylor_thm_bound}
 \| K_\sIR^{\ast(2\oo+\ro),\otimes(1+m)} \ast \fX_\ro^a(v^b,V^b)\|_{\cV^m} \leq c\, C\, [\sIR]^{|a|}.
\end{equation}
\end{thm}
\begin{proof}
The theorem follows from Def.~\ref{dfn:map_X} of the map $\fX^a_\ro$, Lemma~\ref{lem:taylor_bounds_aux} and the bound $\|\partial^c K^{\ast\ro}_\sIR\|_\cK\lesssim |\sIR|^{-|c|}$, $c\in\frM$, $|c|\leq\ro$, proved in Lemma~\ref{lem:kernel_simple_fact}~(A). 
\end{proof}

\begin{thm}\label{thm:taylor}
Let $\ro\in\bN_+$ and $V\in\cD^m$ such that $\cX^{m,b} V \in \cD^m$ for all $b\in\frM^m$, $|b|\leq\ro$. Then
$
 \cX^{m,a} V = \fX_\ro^a(\fI(\cX^{m,b} V),\cX^{m,b} V)
$
for all $a\in\frM^m$, $|a|<\ro$.
\end{thm}
\begin{proof}
First recall that for all $\varphi\in C^\infty(\bR^N)$, $N\in\bN_+$, it holds that
\begin{multline}
 \varphi(y) = \sum_{|b|<\ro} \frac{1}{b!}(y-x)^b\,(\partial^b \varphi)(x) 
 \\
 + \sum_{|b|=\ro} \frac{|b|}{b!}\,(y-x)^b \int_0^1 (1-\tau)^{|b|-1}\,(\partial^b \varphi)(x+\tau(y-x))\,\rd\tau
\end{multline} 
by the Taylor theorem, where the sums are over $b\in\bN_0^N$. Consequently, for all $\varphi\in \sS(\bM^{1+m})$ we have
\begin{multline}
 (\cX^{m,a}\varphi)(x;y_1,\ldots,y_m) 
 = \sum_{|a+b|<\ro} \frac{(-1)^{|b|}}{b!}\cX^{m,a+b}(x;y_1,\ldots,y_m) \,(\partial^b \varphi)(x;x,\ldots,x) 
 \\
 + \sum_{|a+b|=\ro} \frac{(-1)^{|b|}\,|b|}{b!}\cX^{m,a+b}(x;y_1,\ldots,y_m) \int_0^1 (1-\tau)^{|b|-1}\,(\fZ_{\tau}^\dagger\partial^b\varphi)(x;y_1,\ldots,y_m)\,\rd\tau,
\end{multline}
where the map $\fZ_{\tau}^\dagger:=\tau^{\rdim m}\fZ_{1/\tau}$ is a formal dual to the map $\fZ_\tau$ and the sums are over $b\in\frM$. To complete the proof we test $\cX^{m,a} V$ with $\varphi\in \sS(\bM^{1+m})$, apply the above formula and use the definitions of the maps $\fX_\ro^a$ and $\fI$.
\end{proof}

\begin{lem}\label{lem:taylor_bounds_aux}
Let $\oooo\in\bN_0$. The following bounds hold uniformly for $\tau,\sIR\in(0,1]$ and $v\in\cV$, $V\in\cV^m$:
\begin{enumerate}
\item[(A)] $\|K^{\ast\oooo,\otimes(1+m)}_\sIR\ast \fL^m v\|_{\cV^m} \lesssim  \|K^{\ast\oooo}_\sIR\ast v\|_\cV$,
\item[(B)]
$\|K^{\ast2\oooo,\otimes(1+m)}_\sIR\ast\fZ_\tau V\|_{\cV^m}
 \lesssim
 \|K^{\ast\oooo,\otimes(1+m)}_\sIR\ast V\|_{\cV^m}$.
\end{enumerate}
\end{lem}
\begin{proof}
By the exact scaling of the norms and kernels it is enough to prove the lemma for $\sIR=1$ and all $v\in C_\rb(\bM)$, $V\in C(\bM^{1+m})$ such that $\|V\|_{\cV^m}<\infty$. Using $\fP^{\oooo} K^{\ast\oooo}=\delta_\bM$ we obtain
\begin{multline}
 (K^{\ast\oooo,\otimes(1+m)}\ast\fL^m v)(x;y_1,\ldots,y_m)
 \\
 =\int_\bM
 (K^{\ast\oooo}\ast v)(z)~
 \fP^{\oooo}(\partial_z) K^{\ast\oooo}(x-z)K^{\ast\oooo}(y_1-z)\ldots K^{\ast\oooo}(y_m-z)
 \,\rd z
 \\
 =\int_\bM
 (K^{\ast\oooo}\ast v)(z)~
 H^{\ast\oooo}_0(x-z,y_1-z,\ldots,y_m-z)
 \,\rd z,
\end{multline}
where $H_0\in\cK^{1+m}$ is defined in the lemma below and $\fP^{\oooo}(r)=(1-r^2)^\oooo$. This proves the bound (A) since
$
 K^{\ast\oooo,\otimes(1+m)}\ast \fL^m v = H_0^{\ast\oooo}\ast \fL^m(K^{\ast\oooo}\ast v).
$
The bound~(B) follows from the identities
\begin{equation}\label{eq:lem_taylor_useful_identities}
\begin{gathered}
 K^{\otimes(1+m)}
 =(K\otimes\check K^{\otimes m}_\tau) \ast \fZ_\tau(\delta_\bM\otimes K^{\otimes m}),
 \\
 K^{\otimes (1+m)}
 = H_\tau\ast\fZ_\tau(K\otimes\delta_\bM^{\otimes m}),
\end{gathered}
\end{equation}
where the kernels $\check K_\tau$ and $H_\tau$ are defined in the lemma below. Indeed, applying the above equalities recursively and using $\fZ_\tau H\ast \fZ_\tau V=\fZ_\tau(H\ast V)$ we arrive at
\begin{equation}\label{eq:taylor_S_hat_property}
 K^{\ast2\oooo,\otimes(1+m)}\ast\fZ_\tau V
 =
 \check H^{\ast\oooo}_\tau\ast \fZ_\tau(K^{\ast\oooo,\otimes(1+m)}\ast V),
 \quad  
 \check H_\tau:=(K\otimes \check K^{\otimes m}_\tau) \ast H_\tau,
\end{equation}
which implies Part (B) since $\|\check H_\tau\|_{\cK^{1+m}}\lesssim 1$ uniformly in $\tau\in(0,1]$ by the lemma below. It remains to establish the identities~\eqref{eq:lem_taylor_useful_identities}. The first one follows from the fact that $\fZ_\tau(\delta_\bM\otimes K^{\otimes m}) = \delta_\bM\otimes\breve K_\tau^{\otimes m}$, where $\breve K_\tau(x):=\tau^{-\rdim}K(x/\tau)$. To show the second one we observe that by the definition of $\check H_\tau$ given in the lemma below
\begin{equation}
 \fZ_{1/\tau}(H_\tau)(x;y_1,\ldots,y_m)
 =
 \tau^{\rdim m}\, \fP(\partial_x)\, K(x) 
 K(x+\tau(y_1-x)) \ldots K(x+\tau(y_1-x)).
\end{equation}
After convolving both sides with the kernel $K\otimes\delta_\bM^{\otimes m}$ we obtain
\begin{equation}
 \fZ_{1/\tau}(H_\tau)\ast\big(K\otimes\delta_\bM^{\otimes m}\big)
 =
 \fZ_{1/\tau}(K^{\otimes (1+m)})
\end{equation}
The second identity in \eqref{eq:lem_taylor_useful_identities} follows after applying the map $\fZ_\tau$ to both sides of the above equality.
\end{proof}

\begin{lem}\label{lem:taylor_aux}
The distributions $\check K_\tau\in\sS'(\bM)$ and $H_\tau\in\sS'(\bM^{1+m})$ defined by
\begin{equation}
\begin{gathered}
\check K_\tau(x):= \fP(\tau \partial_x) K(x) = \tau^2 \delta_\bM(x) + (1-\tau^2) K(x),
 \\ 
 H_\tau(x;y_1,\ldots,y_m)
 :=
 \fP(\partial_x+(1-\tau)\partial_y) \, K(x) K(y_1) \ldots K(y_m), 
\end{gathered}
\end{equation}
where $\partial_y:=\partial_{y_1}+\ldots+\partial_{y_m}$ and $\fP(r)=1-r^2$, are polynomials in $\tau>0$ of degree~$2$ whose coefficients belong to $\cK$ and $\cK^{1+m}$, respectively.
\end{lem}

\section{Deterministic bounds for the relevant coefficients}\label{sec:deterministic_relavant}

\begin{lem}
For every $\oo\in\bN_0$, $a\in\frM$ and $r\in\bN_0$ the bound
\begin{equation}
 \|\fP^{\oo}_\sIR \cX^a\partial_\uv^r \partial_\sIR G_{\uv,\sIR}\|_\cK \lesssim [\uv]^{(\varepsilon-\sigma)r} [\sIR]^{|a|-\varepsilon r}
\end{equation}  
holds uniformly in $\uv,\sIR\in(0,1/2]$.
\end{lem}
\begin{rem}
The above lemma can be easily proved using the method of the proof of Lemma~\ref{lem:kernel_G}. For $\sigma\notin2\bN_+$ and $|a|\geq\sigma$ the bound stated in the lemma does not hold uniformly for all $\sIR\in(0,1]$ because of the slow decay of $G(x)$ at infinity. For this reason the range $\sIR\in(1/2,1]$ is treated separately in Sec.~\ref{sec:long_range}.
\end{rem}

\begin{thm}\label{thm:deterministic_taylor}
Fix $R>1$. Assume that for all $i,m\in\bN_0$, $a\in\frM^m$ such that $\varrho(i,m)+|a|\leq0$ there exists $\oo\in\bN_0$ such that
\begin{equation}
 \|K^{\ast\oo}_\sIR\ast \partial_\uv^r f^{i,m,a}_{\uv,\sIR}\|_\cV \leq R\,[\uv]^{(\varepsilon-\sigma) r}[\sIR]^{\varrho_{3\varepsilon}(i,m)+|a|}
\end{equation}
for all \mbox{$r\in\{0,1\}$} and $\uv,\sIR\in(0,1/2]$. Then for all $i,m\in\bN_0$, $a\in\frM^m$ there exists $\oo\in\bN_0$ such that
\begin{equation}
 \|K_{\sIR}^{\ast\oo,\otimes(1+m)}\ast \partial_\uv^r\partial_\sIR^s F_{\uv,\sIR}^{i,m,a}\|_{\cV^m}\lesssim R^{1+3i -m}\,[\uv]^{(\varepsilon-\sigma)r}\, [\sIR]^{\varrho_{3\varepsilon}(i,m)+|a|-\sigma s}
\end{equation}
for all $r,s\in\{0,1\}$ uniformly in $\uv,\sIR\in(0,1/2]$. The constants of proportionality in the above bounds depend only on $i,m\in\bN_+$, $a\in\frM^m$ and are otherwise universal. 
\end{thm}
\begin{rem}
Let us remind the reader that the coefficients $F^{i,m,a}_{\uv,\sIR}$ and $f^{i,m,a}_{\uv,\sIR}$ such that $\varrho(i,m)+|a|\leq0$ are called relevant and $F^{i,m,a}_{\uv,\sIR}$ and $f^{i,m,a}_{\uv,\sIR}$ vanish unless $m\leq 3i$. There only finitely many non-zero relevant coefficients $F^{i,m,a}_{\uv,\sIR}$ or $f^{i,m,a}_{\uv,\sIR}$. Recall that the collection of the relevant coefficients $f^{i,m,a}_{\uv,\sIR}$ such that $m\leq 3i$ is called the enhanced noise.
\end{rem}

\begin{rem}
Note that the parameter $\oo\in\bN_0$ in Theorem~\ref{thm:deterministic_taylor} as well as in Theorem~\ref{thm:deterministic_long_range} depends on $i,m\in\bN_+$ and $a\in\frM^m$. We shall use these theorems to bound only the relevant effective force coefficients. Since there are  only finitely many non-zero relevant coefficients one can choose $\oo\in\bN_0$ independent on $i,m\in\bN_+$ and $a\in\frM^m$ such that the bounds stated in the above-mentioned theorems hold true for all relevant coefficients.
\end{rem}

\begin{rem}
Theorem~\ref{thm:deterministic_taylor} together with Theorem~\ref{thm:deterministic_long_range} imply that under the assumption of the former theorem there exists $\oo\in\bN_0$ and a universal constant $c>1$ such that the assumption of Theorem~\ref{thm:deterministic_convergence} is satisfied with $\tilde R=cR$.
\end{rem}

\begin{proof}
{\it The base case:}
We observe that $F^{0,0,0}_{\uv,\sIR}=f^{0,0,0}_{\uv,\sIR}=\xi$. Hence, for $i=0$, $m=0$ the statement follows from the assumption.

{\it The inductive step:}  Fix some $i_\circ\in\bN_+$, $m_\circ\in\bN_0$ and assume that the statement with $s=0$ is true for all $i,m\in\bN_+$ such that either $i<i_\circ$ or $i=i_\circ$ and $m>m_\circ$. We shall prove the statement for $i=i_\circ$, $m=m_\circ$. As in the proof of Theorem~\ref{thm:deterministic_convergence} the induction hypothesis, the flow equation~\eqref{eq:flow_deterministic_i_m_a}, and
\begin{equation}
 \varrho_{3\varepsilon}(i,m)+|a|-\sigma \leq \varrho_{3\varepsilon}(j,1+k)+|b| + |c| + \varrho_{3\varepsilon}(i-j,m-k)+|d| -\varepsilon,
\end{equation}
\begin{equation}
 R^{1+3j-k-1} R^{1+3(i-j)-m+k} = R^{1+3i-m}
\end{equation}
imply the statement for $s=1$. 

In order to prove the statement for $s=0$ let us first assume that $a\in\frM^m$ is such that \mbox{$\varrho(i,m)+|a|>0$}. Then $F^{i,m,a}_{\uv,0}=F^{i,m,a}_\uv=0$. Consequently, the analogs of Eq.~\eqref{eq:determinisitc_intagrate} and the bound~\eqref{eq:determinisitc_bound_u_v} imply that
\begin{equation}
 \|K_{\sIR}^{\ast\oo,\otimes(1+m)}\ast\partial_\uv^r F_{\uv,\sIR}^{i,m,a}\|_{\cV^m}
 \\
 \leq
 \int_{0}^{\sIR}
 \|K_{\uIR}^{\ast\oo,\otimes(1+m)}\ast
 \partial_\uv^r\partial_\uIR F_{\uv,\uIR}^{i,m,a}\|_{\cV^m}\,\rd\uIR.
\end{equation}
The statement follows now from the bound for $\partial_\uv^r\partial_\uIR F_{\uv,\uIR}^{i,m,a}$ and the fact that \mbox{$\varrho_\varepsilon(i,m)+|a|-\varepsilon>0$} if \mbox{$\varrho(i,m)+|a|>0$} by Remark~\ref{rem:rho}.

Now suppose that $a\in\frM^m$ is such that \mbox{$\varrho(i,m)+|a|\leq0$}. Then our bound for $\|K_{\uIR}^{\ast\oo,\otimes(1+m)}\ast \partial_\uv^r\partial_\uIR F_{\uv,\uIR}^{i,m,a}\|_{\cV^m}$ is not integrable at $\uIR=0$. Consequently, the strategy used in the previous paragraph does not work. If $m=0$, then $a=0$ and $\partial_\uv^r F^{i,0,0}_{\uv,\sIR}=\partial_\uv^r f^{i,0,0}_{\uv,\sIR}$. As a result, the statement follows from the assumption. If $m\in\bN_+$, then we use the identity
\begin{equation}\label{eq:thm_deterministic_relevant}
 \partial_\uv^r F^{i,m,a}_{\uv,\sIR}
 =
 \fX^a_\ro(\partial_\uv^r f^{i,m,b}_{\uv,\sIR},
 \partial_\uv^r F^{i,m,b}_{\uv,\sIR}),
\end{equation}
which follows from Theorem~\ref{thm:taylor}. We choose $\ro\in\bN_+$ to be the smallest positive integer such that $\varrho(i,m)+\ro>0$. With this choice of $\ro$ the RHS of the above equality involves only the coefficients $f^{i,m,b}_{\uv,\sIR}$ such that $\varrho(i,m,b)\leq0$, which satisfy the assumed bound, and the coefficients $F^{i,m,b}_{\uv,\sIR}$ such that $\varrho(i,m,b)>0$, for which the statement of the theorem has already been established. To finish the proof of the inductive step we apply Theorem~\ref{thm:taylor_bounds} with $C\lesssim [\uv]^{(\varepsilon-\sigma)r}[\sIR]^{\varrho_{3\varepsilon}(i,m)}$.
\end{proof}

\section{Deterministic bounds for the long-range part of the coefficients}\label{sec:long_range}

\begin{thm}\label{thm:deterministic_long_range}
Fix $R>1$. Assume that for all $i,m\in\bN_0$ there exists $\oo\in\bN_0$ such that
\begin{equation}
 \|K_{\sIR}^{\ast\oo,\otimes(1+m)}\ast \partial_\uv^r\partial_\sIR^s F_{\uv,\sIR}^{i,m}\|_{\cV^m}\lesssim R^{1+3i-m}\,[\uv]^{(\varepsilon-\sigma)r}\, [\sIR]^{\varrho_{3\varepsilon}(i,m)-\sigma s}
\end{equation}
for $r,s\in\{0,1\}$ uniformly in $\uv,\sIR\in(0,1/2]$. Then for all $i,m\in\bN_0$ there exists $\oo\in\bN_0$ such that  the above bound is true for $r,s\in\{0,1\}$ uniformly in $\uv\in(0,1/2]$, $\sIR\in(0,1]$. The constants of proportionality in the above bounds depend only on $i,m\in\bN_+$ and are otherwise universal. 
\end{thm}
\begin{proof}
{\it The base case:}
We observe that $F^{0,0}_{\uv,\sIR}=f^{0,0,0}_{\uv,\sIR}=\xi$ is independent of $\uv,\sIR$. The assumed bound and Lemma~\ref{lem:kernel_u_v} imply that
\begin{equation}
 \|K_\sIR^{\ast\oo}\ast\xi\|_\cV\leq \|K_{\sIR/2}^{\ast\oo}\ast\xi\|_\cV \lesssim [\sIR]^{\varrho_{3\varepsilon}(0,0)}
\end{equation}
uniformly in $\sIR\in(0,1]$. This proves the statement for $i=0$ and $m=0$. 

{\it The induction step:} Fix some $i_\circ\in\bN_+$, $m_\circ\in\bN_0$ and assume that the statement with $s=0$ is true for all $i,m\in\bN_+$ such that either $i<i_\circ$ or $i=i_\circ$ and $m>m_\circ$. We shall prove the statement for $i=i_\circ$, $m=m_\circ$. The induction hypothesis and the flow equation imply the statement for $s=1$. In order to prove the statement for $s=0$ we use the identity
\begin{equation}
 F^{i,m}_{\uv,\sIR} = F^{i,m}_{\uv,1/2} + \int_{1/2}^\sIR \partial_\uIR F^{i,m}_{\uv,\uIR}\,\rd\uIR,
\end{equation}
which by Lemma~\ref{lem:kernel_u_v} implies that
\begin{multline}
 \|K_{\sIR}^{\ast\oo,\otimes(1+m)}\ast \partial_\uv^r F_{\uv,\sIR}^{i,m}\|_{\cV^m}
 \\
 \leq
 \|K_{1/2}^{\ast\oo,\otimes(1+m)}\ast \partial_\uv^r F_{\uv,1/2}^{i,m}\|_{\cV^m}
 +
 \int_{1/2}^{\sIR}
 \|K_{\uIR}^{\ast\oo,\otimes(1+m)}\ast \partial_\uv^r\partial_\uIR F_{\uv,\uIR}^{i,m}\|_{\cV^m}\rd\uIR
\end{multline}
for $\sIR\in[1/2,1]$. This completes the proof of the inductive step.
\end{proof}

\section{Deterministic construction of the solution}\label{sec:solution}

Assuming the bound~\eqref{eq:bound_deterministic_main} for all $r,s\in\{0,1\}$, $i,m\in\bN_0$, $\uv\in(0,1/2]$, $\sIR\in(0,1]$ with fixed $\oo\in\bN_0$, $\tilde R>1$ we show how to construct a solution of Eq.~\eqref{eq:intro_mild}.

\begin{dfn}
The families of sets $\cV_{\sIR}$, $\cB_{\sIR}$, $\sIR\in(0,1]$, are defined as
\begin{equation}
\begin{aligned}
 \cV_{\sIR}&:=\{K^{\ast\oo}_\sIR\ast\varphi\in\cV\,|\,\varphi\in\cV\},
 \\
 \cB_{\sIR}&:=\{K^{\ast\oo}_\sIR\ast\varphi\in\cV\,|\,\varphi\in\cV,\,\|\varphi\|_\cV< \tilde R^2\,[\sIR]^{-\dim(\varPhi)-9\varepsilon}\}.
\end{aligned} 
\end{equation}
By Lemma~\ref{lem:kernel_u_v}, if $\varphi\in\cB_{\sIR}$, then $\varphi\in\cB_{\uIR}$ for all $\uIR$ in a sufficiently small neighbourhood of $\sIR$. For $\uv\in(0,1]$, $\sIR\in[0,1]$ and $|\lambda|\leq \tilde R^{-6}$ we define the functionals
\begin{equation}
 F_{\uv,\sIR}\,:\, \cB_{\sIR}\to \cD,
 \qquad
 \rD F_{\uv,\sIR}\,:\, \cB_{\sIR}\times \cV_{\sIR} \to \cD
\end{equation}
by the formulas
\begin{equation}\label{eq:deterministic_eff_force_series}
 \langle F_{\uv,\sIR}[\varphi],\psi\rangle
 :=\sum_{i=0}^\infty \sum_{m=0}^\infty \lambda^i\,\langle F^{i,m}_{\uv,\sIR},\psi\otimes\varphi^{\otimes m}\rangle,
\end{equation}
\begin{equation}\label{eq:deterministic_D_eff_force_series}
 \langle \rD F_{\uv,\sIR}[\varphi,\zeta],\psi\rangle
 :=\sum_{i=0}^\infty \sum_{m=0}^\infty \lambda^i\,(1+m)\,
 \langle F^{i,1+m}_{\uv,\sIR},\psi\otimes\zeta\otimes\varphi^{\otimes m}\rangle,
\end{equation}
where $\psi\in \sS(\bM)$, $\varphi\in\cB_{\sIR}$ and $\zeta\in\cV_{\sIR}$.
\end{dfn}

\begin{lem}\label{lem:functionals}
For $|\lambda|\leq \tilde R^{-6}$ the functionals $F_{\uv,\sIR}$, $\rD F_{\uv,\sIR}$ are well defined and satisfy the bounds
\begin{equation}
\begin{gathered}
 \|K_\sIR^{\ast\oo}\ast \partial_\uv^r F_{\uv,\sIR}[\varphi]\|_\cV \leq \tilde R\, [\uv]^{(\varepsilon-\sigma)r}\,[\sIR]^{-\dim(\xi)-3\varepsilon}
 \\
 \|K_\sIR^{\ast\oo}\ast \partial_\uv^r \rD F_{\uv,\sIR}[\varphi,\zeta]\|_\cV \leq \tilde R^{-1}\, [\uv]^{(\varepsilon-\sigma)r}\,[\sIR]^{-\dim(\xi)-3\varepsilon}\,\|\fP^\oo_\sIR\zeta\|_\cV
\end{gathered}
\end{equation}
for $r\in\{0,1\}$ and $\uv\in(0,1/2]$, $\sIR\in(0,1]$ and $\varphi\in\cB_\sIR$, $\zeta\in\cV_\sIR$. The functional $\rD F_{\uv,\sIR}$ is the directional derivative of $F_{\uv,\sIR}$. Moreover, given $\uIR\in(0,1]$ and $\varphi\in\cB_\uIR$ the flow equation~\eqref{eq:intro_flow_eq} holds for all $\sIR$ in a sufficiently small neighbourhood of $\uIR$.
\end{lem}
\begin{proof}
The proof is straightforward. For example, using Def.~\ref{dfn:sV} of the space $\cV_m$ and the bounds for the effective force coefficients we obtain
\begin{multline}
 \|K_\sIR^{\ast\oo}\ast \partial_\uv^r F_{\uv,\sIR}[K_\sIR^{\ast\oo}\ast\varphi]\|
 \\
 \leq \sum_{i=0}^\infty\sum_{m=0}^\infty 
 \frac{|\lambda|^i \tilde R^{1+2(3i-m)}}{4(1+i)^2\,4(1+m)^{2}}\, [\uv]^{(\varepsilon-\sigma)r}\, [\sIR]^{\varrho_\varepsilon(i,m)} ~\tilde R^{2m}\, [\sIR]^{-m(\dim(\varPhi)+9\varepsilon)}
\end{multline}
for $\varphi\in\cV$ such that $\|\varphi\|_\cV\leq \tilde R^2\,[\sIR]^{-\dim(\varPhi)-9\varepsilon}$, which implies the desired bound. The flow equation~\eqref{eq:intro_flow_eq} follows from the flow equation~\eqref{eq:intro_flow_eq_i_m}.
\end{proof}

\begin{lem}\label{lem:solution}
For $|\lambda|\leq \tilde R^{-6}$, $r\in\{0,1\}$ and $\uv\in(0,1/2]$, $\sIR\in(0,1]$, it holds that:
\begin{itemize}
 \item[(A)] $\|K_\sIR^{\ast\oo}\ast \partial_\uv^r F_{\uv,1}[0]\|_{\cV} \leq 2^r\tilde R\,[\uv]^{r(\varepsilon-\sigma)}[\sIR]^{-\dim(\xi)-3\varepsilon}$,
 \item[(B)] $\|\fP_\sIR^\oo \partial_\uv^r(G_{\uv,\sIR}\ast F_{\uv,1}[0])\|_{\cV} \leq 3^r/3\,\tilde R^2\,[\uv]^{r(\varepsilon-\sigma)}[\sIR]^{-\dim(\varPhi)-9\varepsilon}$,
 \item[(C)] $F_{\uv,1}[0] = F_{\uv,\sIR}[G_{\uv,\sIR}\ast F_{\uv,1}[0]]$,
 \item[(D)] $\partial_\uv F_{\uv,1}[0]
 =
 (\partial_\uv F_{\uv,\sIR})[G_{\uv,\sIR}\ast F_{\uv,1}[0]] 
 + \rD F_{\uv,\sIR}[G_{\uv,\sIR}\ast F_{\uv,1}[0],\partial_\uv (G_{\uv,\sIR}\ast F_{\uv,1}[0])]$.
\end{itemize}
\end{lem}
\begin{proof}
By Lemma~\ref{lem:functionals} and the equality $G_{\uv,1}=0$ the statement holds true for $\sIR=1$. Chose $\nu\in(0,1]$ and assume that the statement holds true for all $\sIR\in[\nu,1]$. Then by Lemma~\ref{lem:kernel_u_v} we have
\begin{equation}
 \|K_{\uIR}^{\ast\oo}\ast \partial_\uv^r F_{\uv,1}[0]\|_{\cV}
 \leq
 2^{1+r}\, \|K_{\sIR}^{\ast\oo}\ast \partial_\uv^r F_{\uv,1}[0]\|_{\cV}
 \leq 2^{1+r}\tilde R\,[\uv]^{r(\varepsilon-\sigma)}[\sIR]^{-\dim(\xi)-3\varepsilon}
\end{equation}
for all $\uIR\in[\tau\sIR,1]$, $\sIR\in[\nu,1]$ for some $\tau\in(0,1)$ depending only on $\tilde R$. As a result,  
\begin{equation}
 \|K_{\sIR}^{\ast\oo}\ast \partial_\uv^r F_{\uv,1}[0]\|_{\cV} \leq 2^{1+r}\tilde R\,[\uv]^{r(\varepsilon-\sigma)}[\sIR]^{-\dim(\xi)-3\varepsilon}
\end{equation}
for all $\sIR\in[\tau\nu,1]$. This together with Remark~\ref{rem:tilde_R} and the estimate
\begin{equation}
 \|\fP_\sIR^\oo \partial_\uv^r(G_{\uv,\sIR}\ast F_{\uv,1}[0])\|_{\cV} \leq 
 \sum_{u+v=r}\int_\sIR^1 \|\fP_\uIR^{2\oo}\partial_\uv^u\partial_\uIR G_{\uv,\uIR}\|_\cK\, 
 \|K_{\uIR}^{\ast\oo}\ast \partial_\uv^v F_{\uv,1}[0]\|_{\cV}\,\rd\uIR
\end{equation}
implies Part (B) for all $\sIR\in[\tau\nu,1]$ and shows that $G_{\uv,\sIR}\ast F_{\uv,1}[0]\in\cB_\sIR$. Using Lemma~\ref{lem:functionals} and the reasoning presented in Sec.~\ref{sec:intro_flow} below Eq.~\eqref{eq:intro_stationary_relation} we prove that Parts (C) and (D) hold for all $\sIR\in[\tau\uv,1]$. Part (A) for all $\sIR\in[\tau\uv,1]$ follows now from Lemma~\ref{lem:functionals}. To complete the proof of the theorem it is enough to apply the above reasoning recursively with $\nu=\tau^n$, $n\in\bN_0$.
\end{proof}

\begin{thm}\label{thm:solution}
Fix $R>1$. Assume that for all $i,m\in\bN_+$, $a\in\frM^m$ such that $\varrho(i,m)+|a|\leq0$ there exists $\oo\in\bN_0$ such that
\begin{equation}
 \|K^{\ast\oo}_\sIR\ast \partial_\uv^r f^{i,m,a}_{\uv,\sIR}\|_\cV \leq R\,[\uv]^{(\varepsilon-\sigma) r}[\sIR]^{\varrho_{3\varepsilon}(i,m)+|a|}
\end{equation}
for all \mbox{$r\in\{0,1\}$} and $\uv\in(0,1/2]$, $\sIR\in(0,1]$. There exists $\oo\in\bN_0$ and a universal constant $c>1$ such that for $\tilde R=c\,R$, $|\lambda|<\tilde R^{-6}$ and $\uv\in(0,1/2]$ the function $\varPhi_\uv:=G_\uv\ast F_{\uv,\sIR}[0]$ is well defined, solves Eq.~\eqref{eq:intro_mild} and satisfies the bound
\begin{equation}
 \|K_\sIR^{\ast\oo}\ast \partial_\uv^r \varPhi_\uv\|_{\cV} \leq \tilde R^2\, [\uv]^{(\varepsilon-\sigma) r}[\sIR]^{-\dim(\varPhi)-9\varepsilon}
\end{equation}
for all $\uv,\sIR\in(0,1]$.
\end{thm}
\begin{proof}
By Theorems~\ref{thm:deterministic_convergence},~\ref{thm:deterministic_taylor} and~\ref{thm:deterministic_long_range} the assumption implies that there exists $\oo\in\bN_0$ and a universal constant $c>1$ such that the bound~\eqref{eq:bound_deterministic_main} holds true for all $r,s\in\{0,1\}$, $i,m\in\bN_0$ and $\uv\in(0,1/2]$, $\sIR\in(0,1]$ with $\tilde R=c\,R$. The function $\varPhi_\uv$ is well-defined by Lemma~\ref{lem:functionals}. To conclude, we observe that
\begin{equation}
 \|K_\sIR^{\ast\oo}\ast \partial_\uv^r \varPhi_\uv\|_{\cV}
 \leq
 \sum_{u+v=r} \int_0^1 \|\fP^\oo_\uIR \partial_\uv^u \partial_\uIR G_{\uv,\uIR}\|_\cK\, 
 \|K_\uIR^{\ast\oo}\ast K_\sIR^{\ast\oo}\ast \partial_\uv^v F_{\uv,1}[0]\|_{\cV}\,\rd \uIR,
\end{equation}
use Remark~\ref{rem:tilde_R} and apply Lemma~\ref{lem:solution} with $\sIR$ replaced by $\uIR\vee\sIR$.
\end{proof}

\section{Cumulants of the effective force coefficients}\label{sec:cumulants_estimates}

In the remaining part of the paper we show that the assumption of Theorem~\ref{thm:solution} is satisfied for some random $R$ such that $\bE R^{n}<\infty$ for all $n\in\bN_+$. This will follow from the bounds for the joint cumulants of the effective force coefficients $F^{i,m,a}_{\uv,\sIR}$ proved in Sec.~\ref{sec:cumulants_uniform_bounds} and the probabilistic estimates proved in Sec.~\ref{sec:probabilistic}. The bounds for the joint cumulants of $F^{i,m,a}_{\uv,\sIR}$ involve the norm $\|\Cdot\|_{\cV^{\mathsf{m}}}$ introduced in Sec.~\ref{sec:topology_cumulants}. We prove the above-mentioned bounds using a certain flow equation given in Sec.~\ref{sec:flow_equation_cumulants}. We consider joint cumulants of $F^{i,m,a}_{\uv,\sIR}$ instead of $f^{i,m,a}_{\uv,\sIR}$ because there is no flow equation for the joint cumulants of $f^{i,m,a}_{\uv,\sIR}$.

\begin{dfn}
Let $p\in\bN_+$, $I=\{1,\ldots,p\}$ and $\zeta_q$, $q\in I$, be random variables. The joint cumulant~\cite{peccati2011wiener} of the multi-set $(\zeta_q)_{q\in I}=(\zeta_1,\ldots,\zeta_p)$ is defined by
\begin{equation}
 \llangle \zeta_1,\ldots,\zeta_p\rrangle
 \equiv
 \llangle(\zeta_q)_{q\in I}\rrangle
 = 
 (-\ri)^p \partial_{t_1}\ldots\partial_{t_p}\log\bE \exp(\ri t_1 \zeta_1+\ldots+\ri t_p \zeta_p) \big|_{t_1=\ldots=t_p=0}.
\end{equation}
In particular, $\llangle \zeta\rrangle\equiv\bE\zeta$, $\llangle \zeta_1,\zeta_2\rrangle=\llangle \zeta_1\zeta_2\rrangle-\llangle \zeta_1\rrangle\llangle \zeta_2\rrangle$.
\end{dfn}

\begin{lem}\label{lem:cumulants}
Let $p\in\bN_+$, $I=\{1,\ldots,p\}$ and $\zeta_1,\ldots,\zeta_p,\Phi,\Psi$ be random variables. It holds 
\begin{equation}\label{eq:expectation_cumulants}
 \llangle\zeta_1\ldots\zeta_p\rrangle
 =\sum_{r=1}^p
 \sum_{\substack{I_1,\ldots,I_r\subset I,\\I_1\cup\ldots\cup I_r=I\\I_1,\ldots,I_r\neq \emptyset}}
 \llangle(\zeta_q)_{q\in I_1}\rrangle
 \ldots
 \llangle(\zeta_q)_{q\in I_r}\rrangle,
\end{equation} 
\begin{equation}\label{eq:cumulants_product}
 \llangle(\zeta_q)_{q\in I},\Phi\Psi\rrangle
 =
 \llangle(\zeta_q)_{q\in I},\Phi,\Psi\rrangle
 +
 \sum_{\substack{I_1,I_2\subset I\\I_1\cup I_2= I}}
 \llangle(\zeta_q)_{q\in I_1},\Phi\rrangle
 ~
 \llangle(\zeta_q)_{q\in I_2},\Psi\rrangle.
\end{equation}
\end{lem}
\begin{rem}
For the proof of the above lemma see e.g. Proposition~3.2.1 in~\cite{peccati2011wiener}. 
\end{rem}

\begin{dfn}\label{dfn:notation_cumulants_distributions}
Given $n\in\bN_+$, $I=\{1,\ldots,n\}$, $m_1,\ldots,m_n\in\bN_0$ and random distributions \mbox{$\zeta_q\in\sS'(\bM^{1+m_q})$}, $q\in I$, we define the deterministic distribution
$\llangle (\zeta_q)_{q\in I}\rrangle
 \equiv
 \llangle \zeta_1,\ldots,\zeta_n\rrangle
 \in 
 \sS'(\bM^n\times\bM^{m_1+\ldots+m_n})$ by the equality
\begin{equation}
 \langle\llangle \zeta_1,\ldots,\zeta_n\rrangle,\psi_1\otimes\ldots\otimes\psi_n\otimes\varphi_1\otimes\ldots\otimes\varphi_n\rangle
 :=
 \llangle
 \langle\zeta_1,\psi_1\otimes\varphi_1\rangle,\ldots,\langle\zeta_n,\psi_n\otimes\varphi_n\rangle\rrangle,
\end{equation}
where $\psi_q\in \sS(\bM)$, $\varphi_q\in \sS(\bM^{m_q})$, $q\in I$, are arbitrary. 
\end{dfn}

\begin{dfn}\label{dfn:cumulants_eff_force}
A list $(i,m,a,s,r)$, where $i,m\in\bN_0$, $a\in\frM^m$ and $s\in\{0,1\}$, $r\in\{0,1,2\}$ is called an index. Let $n\in\bN_+$ and 
\begin{equation}\label{eq:list_indices}
 \vI\equiv ((i_1,m_1,a_1,s_1,r_1),\ldots,(i_n,m_n,a_n,s_n,r_n))
\end{equation}
be a list of indices. We set $\rn(\vI):=n$, $\rri(\vI):=i_1+\ldots+i_n$, $\mathsf{m}(\vI):=(m_1,\ldots,m_n)$, $\rrm(\vI):=m_1+\ldots+m_n$, $\ra(\vI):=|a_1|+\ldots+|a_n|$, $\rs(\vI):=s_1+\ldots+s_n$ and $\rr(\vI):=r_1+\ldots+r_n$. We use the following notation for the joint cumulants of the effective force coefficients
\begin{equation}
 E^\vI_{\uv,\sIR}:=
 \llangle \partial_\sIR^{s_1}\partial_\uv^{r_1} F^{i_1,m_1,a_1}_{\uv,\sIR} ,\ldots,
 \partial_\sIR^{s_n}\partial_\uv^{r_n} F^{i_n,m_n,a_n}_{\uv,\sIR}\rrangle\in \sS'(\bM^{\rn(\vI)}\times\bM^{\rrm(\vI)}).
\end{equation}
\end{dfn}

\begin{dfn}\label{dfn:varrho_I}
For $\varepsilon\geq0$ and a list of indices $\vI$ of the form~\eqref{eq:list_indices} we define 
\begin{equation}
 \varrho_\varepsilon(\vI) := \varrho_\varepsilon(i_1,m_1)+|a_1| +\ldots +\varrho_\varepsilon(i_n,m_n)+|a_n|\in\bR.
\end{equation}
We also set $\varrho(\vI):=\varrho_0(\vI)$. The cumulants $E^\vI_{\uv,\sIR}$ such that $\varrho(\vI)+(n-1)\rdim\leq 0$ are called relevant. The remaining cumulants are called irrelevant.
\end{dfn}

\begin{rem}
If $\rri(\vI)=0$, then either $E^\vI_{\uv,\sIR}=0$ or $\rn(\vI)=2$, $\rrm(\vI)=0$, $\ra(\vI)=0$, $\rs(\vI)=0$, $\rr(\vI)=0$. In the latter case $E^\vI_{\uv,\sIR}$ is relevant and coincides with the covariance of the white noise. For $\ri(\vI)>0$ the only relevant cumulants are the expectations of the relevant force coefficients.
\end{rem}

\begin{rem}\label{rem:rho2}
For $\varepsilon>0$ and any list of indices $\vI$ such that $\rrm(\vI)\leq 3\rri(\vI)$ it holds that $\varrho_\varepsilon(\vI)<\varrho(\vI)$. Moreover, \mbox{$\varrho_\varepsilon(\vI)+(\rn(\vI)-1)\rdim>0$} for $\varepsilon\in(0,\varepsilon_\diamond)$ and lists of indices $\vI$ such that $\varrho(\vI)+(\rn(\vI)-1)\rdim>0$, where $\varepsilon_\diamond$ was introduced in Remark~\ref{rem:rho}. Recall that $\varepsilon\in(0,\varepsilon_\diamond/3)$ is fixed (see Remark~\ref{rem:rho}).
\end{rem}

\section{Probabilistic analysis}\label{sec:probabilistic}

\begin{thm}\label{thm:probabilistic_bounds}
Fix $n\in2\bN_+$ such that $\rdim/n<\varepsilon$ and $i,m\in\bN_0$, $a\in\frM^m$ such that $\varrho(i,m)+|a|\leq0$. For $s\in\{0,1\}$, $r\in\{0,1,2\}$ we define the list of indices $\vI=\vI(s,r)=((i,m,a,s,r),\ldots,(i,m,a,s,r))$, $\rn(\vI)=n$. Assume that there exists $\oo\in\bN_0$ such that for $s\in\{0,1\}$, $r\in\{0,1,2\}$ the bound
\begin{multline}\label{eq:probabilistic_thm_assumption}
 \int_{\bT^{n-1}} \int_{\bM^{nm}}
 |(K^{\ast\oo,\otimes(n+nm)}_\sIR\ast E^\vI_{\uv,\sIR})(x_1,\ldots,x_n;\ry_1,\ldots,\ry_n)|\, \rd \ry_1\ldots \rd \ry_n\rd x_2\ldots\rd x_n 
 \\
 \lesssim [\uv]^{(\varepsilon-\sigma)\rr(\vI)}
 [\sIR]^{\varrho_\varepsilon(\vI)-\sigma\rs(\vI)+(n-1)\rdim}
\end{multline}
holds uniformly in $x_1\in\bM$, $\uv,\sIR\in(0,1/2]$. Then there exists $\oo\in\bN_0$ and a random variable $R>0$ such that $\bE R^n<\infty$ and for $r\in\{0,1\}$ the bound
\begin{equation}
 \|
 K_\sIR^{\ast\oo}
 \ast \partial_\uv^r f^{i,m,a}_{\uv,\sIR}\|_{\cV} \leq R\,[\uv]^{(\varepsilon-\sigma)r}\,[\sIR]^{\varrho_{3\varepsilon}(i,m)+|a|}.
\end{equation}
holds for all $\uv,\sIR\in(0,1/2]$.
\end{thm}
\begin{rem}
The function
\begin{equation}
 \int_{\bM^{nm}}
 |(K^{\ast\oo,\otimes(n+nm)}_\sIR\ast E^\vI_{\uv,\sIR})(x_1,\ldots,x_n;\ry_1,\ldots,\ry_n)|\, \rd \ry_1\ldots \rd \ry_n
\end{equation}
is $2\pi$ periodic in all variables and the first integral on the LHS of the bound~\eqref{eq:probabilistic_thm_assumption} is an integral over one period.
\end{rem}

\begin{rem}
The assumption of the above theorem is verified in Theorem~\ref{thm:cumulants}. 
\end{rem}

\begin{proof}
By Remark~\ref{rem:periodic} and Lemma~\ref{lem:kernel_simple_fact} (B), which says that \mbox{$\|\fT K^{\ast\rdim}_\sIR\|_{\cV}\lesssim [\sIR]^{-\rdim}$}, the assumption of the theorem implies that for $\ooo=\oo+\rdim$ the bound
\begin{multline}\label{eq:probabilistic_thm_bound}
 \int_{\bM^{nm}}
 |(K^{\ast\ooo,\otimes(n+nm)}_\sIR\ast E^\vI_{\uv,\sIR})(x_1,\ldots,x_n;\ry_1,\ldots,\ry_n)|\, \rd \ry_1\ldots \rd \ry_n
 \\
 \lesssim  [\uv]^{(\varepsilon-\sigma)\rr(\vI)}
 [\sIR]^{\varrho_\varepsilon(\vI)-\sigma\rs(\vI)}
\end{multline}
holds uniformly in $x_1,\ldots,x_n\in\bM$, $\uv,\sIR\in(0,1/2]$. Using the above bound with $x_1=\ldots=x_n=x$, the relation between the expectation of a product of random variables and their joint cumulants given by Eq.~\eqref{eq:expectation_cumulants}, the equalities $\varrho_\varepsilon(\vI)=n(\varrho_\varepsilon(i,m)+|a|)$, $\rs(\vI)=n s$, $\rr(\vI)=nr$ and 
\begin{equation}
 (K^{\ast\ooo}_\sIR \ast \partial_\sIR^s\partial_\uv^r f^{i,m,a}_{\uv,\sIR})(x)=\int_{\bM^m}\!(K^{\ast\ooo,\otimes(1+m)}_\sIR \ast \partial_\sIR^s\partial_\uv^r F^{i,m,a}_{\uv,\sIR})(x;y_1,\ldots,y_m)\,\rd y_1\ldots\rd y_m
\end{equation}
we obtain for $r\in\{0,1,2\}$, $s\in\{0,1\}$ the bound
\begin{equation}
 \sup_{x\in\bM}|\bE (K^{\ast\ooo}_\sIR \ast \partial_\sIR^s\partial_\uv^r f^{i,m,a}_{\uv,\sIR}(x))^n|
 \lesssim [\uv]^{n(\varepsilon-\sigma)r}\,[\sIR]^{n(\varrho_\varepsilon(i,m)+|a|-\sigma s)}
\end{equation}
uniform in $x\in\bM$, $\uv,\sIR\in(0,1/2]$. Lemma~\ref{lem:expectation_sup} and Lemma~\ref{lem:kernel_simple_fact} (C) imply
\begin{equation}
 \bE \|\partial_\sIR^s\partial_\uv^r (K^{\ast\oooo}_\sIR \ast f^{i,m,a}_{\uv,\sIR})\|^n_{\cV}
 \lesssim [\uv]^{n(\varepsilon-\sigma)r}\,[\sIR]^{n(\varrho_\varepsilon(i,m)+|a|-\sigma s-\varepsilon)}
\end{equation}
with $\oooo=\ooo+\rdim$. The theorem follows now from Lemma~\ref{lem:probabilistic_estimate} applied with
\begin{equation}
 \zeta_{2\uv,2\sIR}= K^{\ast\oooo}_\sIR \ast \partial_\uv^r f^{i,m,a}_{\uv,\sIR},\qquad r\in\{0,1\},
\end{equation}
and $\omega=(\varepsilon-\sigma)r$ and $\rho=\varrho_\varepsilon(i,m)+|a|-2\varepsilon \geq \varrho_{3\varepsilon}(i,m)+|a|$, where the last inequality holds for $m\leq3i$ (otherwise $f^{i,m,a}_{\uv,\sIR}$ vanishes identically).
\end{proof}

\begin{rem}
The advantage of the bound of the form~\eqref{eq:probabilistic_thm_assumption} over~\eqref{eq:probabilistic_thm_bound} is that the former bound contains the extra factor $[\sIR]^{(n-1)\rdim}$ on the RHS, and consequently can be more easily proved by induction using the flow equation for cumulants stated in Sec.~\ref{sec:flow_equation_cumulants} and the equality
$E^\vI_{\uv,\sIR}=E^\vI_{\uv,0} + \int_0^\sIR \partial_\uIR E^{\vI}_{\uv,\uIR}\,\rd\uIR$.
\end{rem}

\begin{lem}\label{lem:expectation_sup}
Let $n\in2\bN_+$. There exists a constant $C>0$ such that for all random fields $\zeta\in L^n(\bT)$ and $\sIR\in(0,1/2]$ it holds that
\begin{equation}
 \bE
 \|K_\sIR^{\ast\rdim} \ast \zeta\|^n_{L^\infty(\bT)}
 \leq C\, [\sIR]^{-\rdim}\,
 \bE \|\zeta\|^n_{L^n(\bT)}.
\end{equation}
\end{lem}
\begin{proof}
Note that $K_\sIR^{\ast\rdim} \ast \zeta = \fT K_\sIR^{\ast\rdim} \star \zeta$, where $\star$ is the convolution in $\bT$ and $\fT K_\sIR^{\ast\rdim}$ is the periodization of $K_\sIR^{\ast\rdim}$ (see Def.~\ref{dfn:periodization}). Using the Young inequality for convolution we obtain
\begin{equation}
 \bE\|K_\sIR^{\ast\rdim}\ast \zeta\|^n_{L^\infty(\bT)}
 \leq
 \|\fT K_\sIR^{\ast\rdim}\|_{L^{n/(n-1)}(\bT)}^n\,
 \bE\|\zeta\|^n_{L^n(\bT)}.
\end{equation} 
The lemma follows now from Lemma~\ref{lem:kernel_simple_fact}~(B).
\end{proof}

\begin{lem}\label{lem:probabilistic_estimate}
Fix $n\in\bN_+$. There exists a universal constant $c>0$ such that if 
\begin{equation}
 \bE \|\partial_\sIR^s\partial_\uv^{u} \zeta_{\uv,\sIR}^{\phantom{l}}\|^n_{L^\infty(\bT)}
 \\
 \leq C\,[\uv]^{n(\omega - \sigma u+\varepsilon u)} [\sIR]^{n(\rho-\sigma s+\varepsilon s)}
\end{equation}
for some differentiable random function $\zeta:(0,1]^2\to L^\infty(\bT)$, some $C>0$, $\omega,\rho\leq0$ and all $u,s\in\{0,1\}$ and $\uv,\sIR\in(0,1]$, then
\begin{equation}
 \bE\Big( \sup_{\uv,\sIR\in(0,1]}
 [\uv]^{-n\omega}[\sIR]^{-n\rho}\,\|\zeta_{\uv,\sIR}\|^n_{L^\infty(\bT)}\Big) \leq c\,C.
\end{equation}
\end{lem}
\begin{proof} 
For all $\uv,\sIR\in(0,1]$ we have
\begin{multline}
 [\uv]^{-\omega} [\sIR]^{-\rho}\,\|\zeta_{\uv,\sIR}\|_{L^\infty(\bT)}
 \leq
 \|\zeta_{1,1}\|_{L^\infty(\bT)}
 \\
 +
 \int_\sIR^1 [\uIR]^{-\rho}\,\|\partial_\uIR\zeta_{1,\uIR}\|_{L^\infty(\bT)}\,\rd\uIR
 +
 \int_\sIR^1 \int_\uv^1 [\nu]^{-\omega}[\uIR]^{-\rho}\,\|\partial_\uIR\partial_\nu \zeta_{\nu,\uIR}\|_{L^\infty(\bT)} \,\rd\nu\rd\uIR
\end{multline}
By the Minkowski inequality we get the bound
\begin{multline}\label{eq:probabilistic_proof_ieq}
 \bE\Big( \sup_{\uv,\sIR\in(0,1]}
 [\uv]^{-n\omega}[\sIR]^{-n\rho}\,\|\zeta_{\uv,\sIR}\|^n_{L^\infty(\bT)}\Big)^{\!\frac{1}{n}}
 \leq
 \bE\Big(\|\zeta_{1,1}\|^n_{L^\infty(\bT)}\Big)^{\!\frac{1}{n}}
 \\
 +
 \int_0^1 [\sIR]^{-\rho} \,\bE\Big( \|\partial_\sIR \zeta_{1,\sIR}\|^n_{L^\infty(\bT)}\Big)^{\!\frac{1}{n}}\rd\sIR
 +
 \int_0^1\int_0^1 [\uv]^{-\omega}[\sIR]^{-\rho}\, \bE\Big(\|\partial_\sIR\partial_\uv \zeta_{\uv,\sIR}\|^n_{L^\infty(\bT)}\Big)^{\!\frac{1}{n}}\rd\uv\rd\sIR,
\end{multline}
which implies the statement.
\end{proof}

\section{Function spaces for cumulants}\label{sec:topology_cumulants}

\begin{dfn}
For $n\in\bN_+$ we say that $V\in C(\bM^n)$ is translationally invariant iff
$V(x_1,\ldots,x_n)=V(x_1+x,\ldots,x_n+x)$ for all $x_1,\ldots,x_n,x\in\bM$.
\end{dfn}

\begin{dfn}\label{dfn:sVm}
Let $n\in\bN_+$, \mbox{$\mathsf{m}=(m_1,\ldots,m_n)\in\bN_0^n$} and $m=m_1+\ldots+m_n$. The vector space $\cV^{\mathsf{m}}_\rt$ consists of translationally invariant $V\in C(\bM^n\times\bM^m)$ such that
\begin{equation}
 \|V\|_{\cV^{\mathsf{m}}}
 :=
 \sup_{x_1\in\bM} \int_{\bT^{n-1}\times\bM^m}
 |V(x_1,\ldots,x_n;y_1,\ldots,y_m)|\,\rd x_2\ldots\rd x_n\,\rd y_1\ldots\rd y_m
\end{equation}
is finite and the function $\fU^{\mathsf{m}} V\in C(\bM^n\times\bM^m)$ defined by
\begin{equation}
 \fU^{\mathsf{m}} V(x_1,\ldots,x_n;y_1,\ldots,y_m):=V(x_1,\ldots,x_n;\mathrm y_1+x_1,\ldots,\mathrm y_n+x_n)
\end{equation} 
is $2\pi$ periodic in variables $x_1,\ldots,x_n$ for every
\begin{equation}
 (y_1,\ldots,y_m)=(\mathrm y_1,\ldots,\mathrm y_n)\in\bM^{m_1}\times\ldots\times\bM^{m_n}=\bM^m,
\end{equation}
where for arbitrary $x\in\bM$ and $\ry=(y_1,\ldots,y_m)\in\bM^m$, $m\in\bN_0$, we use the following notation $\ry+x:=(y_1+x,\ldots,y_n+x)\in\bM^m$. For $\oo\in\bN_0$ the space $\cD_\rt^{\mathsf{m};\oo}$ consists of distributions $V\in \sS'(\bM^n\times\bM^m)$ such that $K^{\ast\oo,\otimes (n+m)}_\sIR\ast V\in\cV^{\mathsf{m}}$. The space $\cD^{\mathsf{m}}_\rt$ is the union of the spaces $\cD_\rt^{\mathsf{m};\oo}$, $\oo\in\bN_0$.
\end{dfn}

\begin{rem}
For $n=1$ and $\mathsf{m}=m$ the norm $\|\Cdot\|_{\cV^{\mathsf{m}}}$ coincides with the norm $\|\Cdot\|_{\cV^m}$ introduced in Def.~\ref{dfn:sV}.
\end{rem}

\begin{rem}\label{rem:translational_permutation}
Using translational invariance of $V\in\cV^{\mathsf{m}}_\rt$ one shows that
\begin{equation}
 \|V\|_{\cV^{\mathsf{m}}}
 =
 \frac{1}{(2\pi)^{\rdim}}\int_{\bT^n\times\bM^m}
 |V(x_1,\ldots,x_n;y_1,\ldots,y_m)|\,\rd x_1\ldots\rd x_n \rd y_1\ldots\rd y_m.
\end{equation} 
\end{rem}

\begin{rem}\label{rem:Vm_K2}
For $V\in\cV^{\mathsf{m}}$ and $K\in\cK^{n+m}$ it holds $\|K\ast V\|_{\cV^{\mathsf{m}}}\leq \|K\|_{\cK^{n+m}} \|V\|_{\cV^{\mathsf{m}}}$.
\end{rem}

\begin{dfn}\label{dfn:map_Ym}
For $V\in\cD^{\mathsf{m}}_\rt$ and a permutation $\pi\in\cP_{m_1}$ we define $\fY_\pi V\in\cD^{\mathsf{m}}_\rt$ by
\begin{equation}
 \big\langle \fY_\pi V,\botimes_{q=1}^n \psi_q\otimes
 \botimes_{q=1}^m \varphi_q\big\rangle
 :=
 \big\langle V,\botimes_{q=1}^n \psi_q\otimes
 \botimes_{q=1}^{m_1} \varphi_{\pi(q)}\otimes
 \botimes_{q=m_1+1}^{m} \varphi_q \big\rangle,
\end{equation}
where $\psi_1,\ldots,\psi_n,\varphi_1,\ldots,\varphi_m\in\ \sS(\bM)$. For $V\in\cD^{\mathsf{m}}_\rt$ and $\omega\in\cP_{n}$ we define $\fY^\omega V\in\cD^{\omega(\mathsf{m})}_\rt$, where $\omega(\mathsf{m}):=(m_{\omega(1)},\ldots,m_{\omega(n)})$, by 
\begin{equation}
 \big\langle \fY^\omega V,\botimes_{q=1}^n \psi_q\otimes
 \botimes_{q=1}^n \varphi_q\big\rangle
 :=
 \big\langle V,\botimes_{q=1}^n \psi_{\omega(q)}\otimes
 \botimes_{q=1}^{n} \varphi_{\omega(q)}\big\rangle,
\end{equation}
where $\psi_q\in\ \sS(\bM)$, $\varphi_q\in \sS(\bM^{m_q})$ for $q\in\{1,\ldots,n\}$.
\end{dfn}
\begin{rem}
The maps $\fY_\pi:\cV^{\mathsf{m}}_\rt\to\cV^{\mathsf{m}}_\rt$, $\fY^\omega:\cV^{\mathsf{m}}_\rt\to\cV^{\omega(\mathsf{m})}_\rt$ are bounded with norm one.
\end{rem}

\begin{dfn}\label{dfn:maps_A_B}
Fix $n\in\bN_+$, $\hat n\in\{1,\ldots,n\}$, $m_1,\ldots,m_{n+1}\in\bN_0$. Let
\begin{equation}
\begin{gathered}
 \mathsf{m}=(m_1+m_{n+1},m_2,\ldots,m_n)\in\bN_0^n,
 \qquad
 \tilde{\mathsf{m}}=(1+m_1,m_2,\ldots,m_{n+1})\in\bN_0^{n+1},
 \\
 \hat{\mathsf{m}}=(1+m_1,m_2,\ldots,m_{\hat n})\in\bN_0^{\hat n},
 \qquad
 \check{\mathsf{m}}=(m_{\hat n+1},\ldots,m_{n+1})\in\bN_0^{n-\hat n+1}.
\end{gathered} 
\end{equation} 
The bilinear map $\fA\,:\,\sS(\bM) \times\cV^{\tilde{\mathsf{m}}}_\rt\to\cV^{\mathsf{m}}_\rt$ is defined by
\begin{multline}\label{eq:fA_dfn}
 \fA(G,V)(x_1,\ldots,x_n;\ry_1,\ry_{n+1},\ry_2,\ldots,\ry_n)
 \\
 :=
 \int_{\bM^2} V(x_1,\ldots,x_{n+1}; y,\ry_1,\ldots,\ry_{n+1})\,G(y-x_{n+1})\,\rd y\rd x_{n+1}.
\end{multline}
The trilinear map $\fB\,:\,\sS(\bM)\times \cV^{\hat{\mathsf{m}}}_\rt\times\cV^{\check{\mathsf{m}}}_\rt\to\cV^{\mathsf{m}}_\rt$ is defined by 
\begin{multline}\label{eq:fB_dfn}
 \fB(G,W,U)(x_1,\ldots,x_{n};\ry_1,\ry_{n+1},\ry_2,\ldots,\ry_{n})
 :=
 \int_{\bM^2} W(x_1,\ldots,x_{\hat n};y,\ry_1,\ldots,\ry_{\hat n})
 \\
 \times\,G(y-x_{n+1})\, 
 U(x_{n+1},x_{\hat n+1},\ldots,x_{n};\ry_{n+1},\ry_{\hat n+1},\ldots,\ry_{n})\,\rd y\rd x_{n+1}.
\end{multline}
In the above equations $\ry_j\in\bM^{m_j}$, $j\in\{1,\ldots,n+1\}$.
\end{dfn}

\begin{rem}
The map $\fB$ is a generalization of the map $\fB$ introduced in Def.~\ref{dfn:map_B}. The maps $\fA$ and $\fB$ appear on the RHS of the flow equation for the cumulants of the  effective force coefficients introduced in Sec.~\ref{sec:flow_equation_cumulants}. 
\end{rem}

\begin{lem}\label{lem:fA_fB_bounds}
The maps $\fA:\sS(\bM) \times\cV_\rt^{\tilde{\mathsf{m}}}\to\cV_\rt^{\mathsf{m}}$,
\mbox{$\fB:\sS(\bM)\times \cV_\rt^{\hat{\mathsf{m}}}\times\cV_\rt^{\check{\mathsf{m}}}\to\cV_\rt^{\mathsf{m}}$} are well defined. It holds that
\begin{equation}\label{eq:fA_ieq}
 \|\fA(G,V)\|_{\cV^{\mathsf{m}}}
 \leq
 \|\fT |G|\|_{\cV}\,
 \|V\|_{\cV^{\tilde{\mathsf{m}}}},
\end{equation}
\begin{equation}\label{eq:fB_ieq}
 \|\fB(G,W,U)\|_{\cV^{\mathsf{m}}}
 \leq
 \|G\|_\cK\,
 \|W\|_{\cV^{\check{\mathsf{m}}}}\, 
 \|U\|_{\cV^{\hat{\mathsf{m}}}},
\end{equation}
where $\fT|G|$ is the periodization of $|G|$ introduced in Def.~\ref{dfn:periodization}, $\|G\|_\cK=\|G\|_{L^1(\bM)}$ and $\|\fT |G|\|_\cV=\|\fT |G|\|_{L^\infty(\bT)}$.
\end{lem}
\begin{proof}
The functions $\fU^{\mathsf{m}}\fA(G,V)$, $\fU^{\mathsf{m}}\fB(G,W,U)$ are translationally invariant and $2\pi$ periodic since
\begin{multline}
 \fU^{\mathsf{m}}\fA(G,V)(x_1,\ldots,x_n;\ry_1,\ry_{n+1},\ry_2,\ldots,\ry_n)
 =
 \int_{\bM^2} G(y-x_{n+1})
 \\\times
 \fU^{\tilde{\mathsf{m}}} V(x_1,\ldots,x_n,x_{n+1}+x_1; y,\ry_1,\ldots,\ry_n,\ry_{n+1}-x_{n+1})\,\rd y\rd x_{n+1},
\end{multline}
\begin{multline}
 \fU^{\mathsf{m}}\fB(G,W,U)(x_1,\ldots,x_{n};\ry_1,\ry_{n+1},\ry_2,\ldots,\ry_{n})
 \\
 =
 \int_{\bM^2} \fU^{\hat{\mathsf{m}}} W(x_1,\ldots,x_{\hat n};y,\ry_1,\ldots,\ry_{\hat n})
 \,G(y-x_{n+1})\, \\
 \times
 \fU^{\check{\mathsf{m}}} U(x_{n+1}+x_1,x_{\hat n+1},\ldots,x_n;\ry_{n+1}-x_{n+1},\ry_{\hat n+1},\ldots,\ry_{n})\,\rd y\rd x_{n+1}.
\end{multline}
To prove the first bound note that
\begin{multline}
 \|\fA(G,V)\|_{\cV^{\mathsf{m}}}
 \leq 
 \sup_{x_1\in\bT}\int_{\bM^2\times\bT^{n-1}\times\bM^m} |G(y+x_1-x_{n+1})|
 \\
 \times |\fU^{\mathsf{m}}V(x_1,\ldots,x_{n+1}; y,\ry_1,\ldots,\ry_{n+1})|\,\rd y\rd x_{n+1}\rd x_2\ldots\rd x_n\rd\ry_1\ldots\rd\ry_{n+1}
\end{multline}
Using periodicity of $\fU^{\mathsf{m}}V$ we arrive at
\begin{multline}
 \|\fA(G,V)\|_{\cV^{\mathsf{m}}}\leq\sup_{x_1\in\bT}\int_{\bM\times\bT^{n}\times\bM^m} (\fT|G|)(y+x_1-x_{n+1})
 \\
 \times |\fU^{\mathsf{m}}V(x_1,\ldots,x_{n+1}; y,\ry_1,\ldots,\ry_{n+1})|\,\rd y\rd x_{n+1}\rd x_2\ldots\rd x_n\rd\ry_1\ldots\rd\ry_{n+1}
\end{multline}
This implies $\|\fA(G,V)\|_{\cV^{\mathsf{m}}}\leq \|\fT |G|\|_{\cV}\, \|\fU^{\mathsf{m}}V\|_{\cV^{\tilde{\mathsf{m}}}}= \|\fT |G|\|_{\cV}\, \|V\|_{\cV^{\tilde{\mathsf{m}}}}$. It holds that
\begin{multline}
 \|\fB(G,W,U)\|_{\cV^{\mathsf{m}}}
 \leq
 \sup_{x_1\in\bT}\int_{\bT^{n-1}\times\bM^{m+2}} |W(x_1,\ldots,x_{\hat n};y,\ry_1,\ldots,\ry_{\hat n})|
 \,|G(y-z)|\,
 \\\times
 |U(z,x_{\hat n+1},\ldots,x_{n};\ry_{n+1},\ry_{\hat n+1},\ldots,\ry_{n})|\,\rd x_2\ldots\rd x_{n}\rd z\rd y\rd\ry_1\ldots\rd\ry_{n+1}
 \\
 \leq 
 \sup_{x_1\in\bT}\int_{\bT^{\hat n-1}\times\bM^{\hat m}} |W(x_1,\ldots,x_{\hat n};y,\ry_1,\ldots,\ry_{\hat n})|
 \,\rd x_2\ldots\rd x_{\hat n}\rd y\rd\ry_1\ldots\rd\ry_{\hat n}
 \\\times
 \sup_{y\in\bM}\int_\bM |G(y-z)|\,\rd z
 ~\|U\|_{\cV^{\check{\mathsf{m}}}},
\end{multline}
where $\hat m:=1+m_1+\ldots+m_{\hat n}$. The second of the stated bounds follows from the above estimate.
\end{proof}

\begin{rem}\label{lem:fA_fB_Ks}
The fact that $\fP_\sIR^\oo K_\sIR^{\ast\oo}=\delta_\bM$ implies that for all $\sIR>0$ it holds that
\begin{equation}
\begin{gathered}
 K_\sIR^{\ast\oo,\otimes(n+m)}\ast\fA(G,V)=
 \fA\big(\fP^{2\oo}_\sIR G,
 K_\sIR^{\ast\oo,\otimes (n+m+2)}\ast V\big),
 \\
 K_\sIR^{\ast\oo,\otimes(n+m)}\ast\fB(G,W,U)=
 \fB\big(\fP^{2\oo}_\sIR G,
 K_\sIR^{\ast\oo,\otimes(\hat n+\hat m+1)}\ast W,
 K_\sIR^{\ast\oo,\otimes(n-\hat n+\check m+1)}\ast U\big),
\end{gathered} 
\end{equation}
where \mbox{$m=m_1+\ldots+m_{n+1}$}, \mbox{$\hat m=m_1+\ldots+m_{\hat n}$} and \mbox{$\check m=m_{\hat n+1}+\ldots+m_{n+1}$}. This allows to define $\fA(G,V)\in\cD^{\mathsf{m}}_\rt$ and $\fB(G,W,U)\in\cD^{\mathsf{m}}_\rt$ for all $G\in\sS(\bM)$, $V\in\cD^{\tilde{\mathsf{m}}}_\rt$, $W\in\cD^{\hat{\mathsf{m}}}_\rt$, $U\in\cD^{\check{\mathsf{m}}}_\rt$. For future reference note that by Lemma~\ref{lem:kernel_simple_fact}~(B)
\begin{equation}
 \|\fT |\fP^{2\oo}_\sIR G|\|_\cV\leq \|\fT K_\sIR^{\ast\rdim}\|_\cV \,\|\fP^{2\oo+\rdim}_\sIR G\|_\cK\lesssim [\sIR]^{-\rdim}\,\|\fP^{2\oo+\rdim}_\sIR G\|_\cK
\end{equation}
uniformly in $\sIR\in(0,1]$ since $K_\sIR^{\ast\rdim}=|K_\sIR^{\ast\rdim}|$. 
\end{rem}

\section{Flow equation for cumulants}\label{sec:flow_equation_cumulants}

\begin{lem}\label{lem:flow_E_general}
Let $n\in\bN_+$, $i_1\in\bN_0$, $m_1,\ldots,m_n\in\bN_0$, $a_1\in\frM^{m_1}$, $r_1\in\{0,1,2\}$ and $I\equiv \{2,\ldots,n\}$. For any random  distributions $\zeta_q\in\cD^{m_q}$, $q\in I$, the cumulant
\begin{equation}
 \llangle 
 \partial_\sIR \partial_\uv^{r_1} F^{i_1,m_1,a_1}_{\uv,\sIR},
 (\zeta_q)_{q\in I}\rrangle \in \cD^{(m_1,\ldots,m_n)}_\rt
\end{equation}
is a linear combination of the expressions
\begin{multline}\label{eq:flow_E_general}
 \sum_{\substack{I_1,I_2\subset I\\I_1\cup I_2= I\\I_1\cap I_2=\emptyset}}
 \fY_\pi\fB\big(\cX^c\partial_\uv^u\partial_\sIR G_{\uv,\sIR},
 \llangle
 \partial_\uv^v F^{j,1+k,b}_{\uv,\sIR},(\zeta_q)_{q\in I_1}\rrangle
 ,\llangle
 \partial_\uv^{w} F^{i_1-j,m_1-k,d}_{\uv,\sIR},(\zeta_q)_{q\in I_2}
 \rrangle\big)
 \\+
 \fY_\pi\fA\big(\cX^c\partial_\uv^u\partial_\sIR G_{\uv,\sIR},
 \llangle
 \partial_\uv^v F^{j,1+k,b}_{\uv,\sIR},
 (\zeta_q)_{q\in I},
 \partial_\uv^w F^{i_1-j,m_1-k,d}_{\uv,\sIR}
 \rrangle\big),
\end{multline}
where $\pi\in\cP_{m_1}$, $j\in\{1,\ldots,i_1\}$, $k\in\{0,\ldots,m_1\}$, $u,v,w\in\bN_0$, $u+v+w=r_1$ and $b,c,d$ are multi-indices satisfying the condition $b+c+d=\pi(a_1)$ whose meaning is explained below Eq.~\eqref{eq:poly_binom}. The coefficients of the above linear combination depend only on $m_1$, $k$, $a,b,c,d$ and $u,v,w$. We used the notation introduced in Def.~\ref{dfn:notation_cumulants_distributions}.
\end{lem}
\begin{proof}
The statement follows immediately from Eqs.~\eqref{eq:flow_deterministic_i_m_a} and \eqref{eq:cumulants_product}.
\end{proof}

\begin{lem}\label{lem:flow_E_form_bound}
Let $\vJ\equiv (\vJ_1,\ldots,\vJ_n)=((i_1,m_1,a_1,s_1,r_1),\ldots,(i_n,m_n,a_n,s_n,r_n))$ be a list of indices such that $s_l=1$ for some $l\in\{1,\ldots,n\}$.
\begin{enumerate}
\item[(A)]
The distribution $E^\vJ_{\uv,\sIR}$ can be expressed as a linear combination of distributions of the form
\begin{equation}
 \fY^\omega\fY_\pi\fA\big(\cX^c\partial_\uv^u\partial_\sIR G_{\uv,\sIR},E^{\vK}_{\uv,\sIR}\big)
 \qquad
 \textrm{or}
 \qquad
 \fY^\omega\fY_\pi\fB\big(\cX^c\partial_\uv^u\partial_\sIR G_{\uv,\sIR},
 E^{\vL}_{\uv,\sIR},
 E^{\vM}_{\uv,\sIR}
 \big),
\end{equation}
where $u\in\{0,\ldots,r_l\}$, $c\in\frM$ is some multi-index, $\vK$, $\vL$, $\vM$ are some lists of indices and $\omega\in\cP_n$, $\pi\in\cP_{m_l}$ are some permutations.

\item[(B)] The lists of indices $\vK$, $\vL$, $\vM$ satisfy the conditions
\begin{equation} 
\begin{split}
&\rn(\vK)=\rn(\vJ)+1,\\
&\rri(\vK)=\rri(\vJ),\\
&\rrm(\vK)=\rrm(\vJ)+1,\\
&\ra(\vK)+|c|=\ra(\vJ),\\
&\rs(\vK)=\rs(\vJ)-1,\\
&\rr(\vK)=\rr(\vJ)-u,\\
&\varrho_\varepsilon(\vJ) -\sigma = \varrho_\varepsilon(\vK)-\rdim -2\varepsilon,
\end{split}
~\textrm{or}\qquad
\begin{split}
&\rn(\vL)+\rn(\vM)=\rn(\vJ)+1,\\
&\rri(\vL)+\rri(\vM)=\rri(\vJ),\\
&\rrm(\vL)+\rrm(\vM)=\rrm(\vJ)+1,\\ 
& \ra(\vL)+\ra(\vM)+|c|=\ra(\vJ),\\
&\rs(\vL)+\rs(\vM)=\rs(\vJ)-1,\\
&\rr(\vL)+\rr(\vM)=\rr(\vJ)-u,\\
&\varrho_\varepsilon(\vJ) -\sigma = \varrho_\varepsilon(\vL)+\varrho_\varepsilon(\vM)-2\varepsilon.
\end{split}
\end{equation}

\item[(C)] Fix some $\oo\in\bN_0$. Suppose that the bound
\begin{equation}
 \|K^{\ast\oo,\otimes(\rn(\vI)+\rrm(\vI))}_\sIR\ast E^\vI_{\uv,\sIR}\|_{\cV^{\mathsf{m}(\vI)}}
 \\
 \lesssim [\uv]^{(\varepsilon-\sigma)\rr(\vI)}
 [\sIR]^{\varrho_\varepsilon(\vI)-\sigma\rs(\vI)+(\rn(\vI)-1)\rdim}
\end{equation}
holds uniformly in $\uv,\sIR\in(0,1/2]$ for all lists of indices \mbox{$\vI\in\{\vK,\vL,\vM\}$}, where $\vK,\vL,\vM$ are arbitrary lists of indices satisfying the conditions specified in Part~(B) given the list of indices $\vJ$. Then the above bound holds uniformly in $\uv,\sIR\in(0,1/2]$ for $\vI=\vJ$.
\end{enumerate}
\end{lem}
\begin{proof}
Without loss of generality we can assume that $l=1$. Part (A) of the lemma follows immediately from Lemma~\ref{lem:flow_E_general} applied with
\begin{equation}
 \zeta_q\equiv \partial_\sIR^{s_q}\partial_\uv^{r_q} F^{i_q,m_q,a_q}_{\uv,\sIR}, 
 \qquad q\in\{2,\ldots,n\}.
\end{equation}
The multi-index $c\in\frM$, $u\in\{0,\ldots,r_1\}$ and the permutation $\pi\in\cP_{m_1}$ in the statement coincide with the respective objects in Eq.~\eqref{eq:flow_E_general}. The permutation $\omega\in\cP_n$ is trivial because of the assumption $l=1$. Moreover, it holds that
\begin{equation}
\begin{gathered}
 \vK=((j,k+1,b,0,v),\,\vJ_2,\ldots,\vJ_n,\,(i_1-j,m_1-k,d,0,w)),
 \\
 \vL= (j,k+1,b,0,v)\sqcup (\vJ_q)_{q\in I_1},
 \qquad
 \vM =(i_1-j,m_1-k,d,0,w) \sqcup (\vJ_q)_{q\in I_2},
\end{gathered} 
\end{equation}
where $\sqcup$ denotes the concatenation of lists, $I_1\cup I_2=I=\{2,\ldots,n\}$, $I_1\cap I_2=\emptyset$ and $j\in\{1,\ldots,i_1\}$, $k\in\{0,\ldots,m_1\}$, $b\in\frM^{1+k}$, $d\in\frM^{m-k}$, $v,w\in\{0,\ldots,r_1\}$ coincide with the respective objects in Eq.~\eqref{eq:flow_E_general}. This implies that the lists $\vK$, $\vL$, $\vM$ satisfy the conditions given in Part (B). The last condition follows from Def.~\ref{dfn:varrho_I} and $\dim(\xi)+\dim(\Phi)=\rdim-\sigma$. To prove Part (C) we use Parts (A), (B), Lemma~\ref{lem:fA_fB_bounds}, Remark~\ref{lem:fA_fB_Ks} and Lemma~\ref{lem:kernel_G} applied with $r\in\{0,1,2\}$.
\end{proof}

\section{Uniform bounds for cumulants}\label{sec:cumulants_uniform_bounds}

\begin{thm}\label{thm:cumulants}
There exists a unique choice of the counterterms $c^{[i]}_\uv$ in Eq.~\eqref{eq:force} such that:
$\bE f^{1,3,0}_{\uv,1/2}=1$, $\bE f^{i,1,0}_{\uv,1/2}=\mathfrak{f}^{[i]}$, $i\in\{1,\ldots,i_\sharp\}$, and $\bE f^{i,m,a}_{\uv,1/2} = 0$ for all other $i\in\bN_+$, $m\in\bN_0$, $a\in\frM^m$ such that $\varrho(i,m)+|a|\leq 0$. Fix arbitrary \mbox{$\mathfrak{f}^{[i]}\in\bR$}, $i\in\{1,\ldots,i_\sharp\}$. For all list of indices
$\vI$ there exists \mbox{$\oo\in\bN_0$} such that the bound
\begin{equation}\label{eq:thm_cumulants}
 \|K^{\ast\oo,\otimes(n+m)}_{\sIR}\ast E^\vI_{\uv,\sIR}\|_{\cV^{\mathsf{m}}}
 \lesssim 
 [\uv]^{(\varepsilon-\sigma)\rr(\vI)}
 [\sIR]^{\varrho_\varepsilon(\vI)-\sigma\rs(\vI)+(n-1)\rdim}
\end{equation}
holds uniformly in $\uv,\sIR\in(0,1/2]$, where $n=\rn(\vI)$, $\mathsf{m}=\mathsf{m}(\vI)$, \mbox{$m=\rrm(\vI)$}. Moreover, the following condition is satisfied
\begin{equation}
 E^\vI_{\uv,\sIR}(x_1,\ldots,x_n;\ry_1,\ldots,\ry_n) = (-1)^{\ra(\vI)}
 E^\vI_{\uv,\sIR}(-x_1,\ldots,-x_n;-\ry_1,\ldots,-\ry_n),
\end{equation}
where $x_j\in\bM$, $\ry_j\in\bM^{m_j}$ for $j\in\{1,\ldots,n\}$ and $E^\vI_{\uv,\sIR}=0$ unless $m+n\in2\bN_0$.
\end{thm}
\begin{rem}
By stationarity $\bE f^{i,m,a}_{\uv,\sIR}(x)$ is a constant. Since $\partial_\sIR F^{1,3}_{\uv,\sIR}=0$ it holds that $\bE f^{1,3,0}_{\uv,1/2}=f^{1,3,0}_{\uv,0}=1$. There is no distinguished value of $\bE f^{i,1,0}_{\uv,1/2}=\mathfrak{f}^{[i]}$, $i\in\{1,\ldots,i_\sharp\}$, ultimately, because there is no distinguished function $\chi$ in Def.~\ref{dfn:kernel_G} of the cutoff propagator $G_\uv$. The vanishing of $\bE f^{i,m,a}_{\uv,1/2}$ for all other $i\in\bN_+$, $m\in\bN_0$, $a\in\frM^m$ such that $\varrho(i,m)+|a|\leq 0$ is enforced by the properties of the cumulants given in the last sentence of the theorem (here we use the assumption $\rdim\in\{1,\ldots,6\}$). Observe that these properties are consequences of the following symmetries of Eq.~\eqref{eq:intro_mild}: $\varPhi(x)\mapsto-\varPhi(x)$, $\xi(x)\mapsto-\xi(x)$ and $\varPhi(x)\mapsto\varPhi(-x)$, $\xi(x)\mapsto\xi(-x)$, which in particular preserve the law of $\xi$. The counterterms $c^{[i]}_\uv$ are related to the renormalization parameters $\mathfrak{f}^{[i]}$ by the formula
\begin{equation}
 c^{[i]}_\uv:=f^{i,1,0}_{\uv}=f^{i,1,0}_{\uv,0} = \mathfrak{f}^{[i]} -\int_0^{1/2} \bE\,\partial_\sIR f^{i,1,0}_{\uv,\sIR}\,\rd\sIR,\qquad i\in\{1,\ldots,i_\sharp\}.
\end{equation}
The constants $\mathfrak{f}^{[i]}\in\bR$, $i\in\{1,\ldots,i_\sharp\}$ parametrize the class of solutions of Eq.~\eqref{eq:intro_mild} constructed in the paper (generically this is an over-parametrization).
\end{rem}

\begin{proof}
We first note that the theorem is trivially true for all list of indices $\vI$ such that $\rrm(\vI)>3\rri(\vI)$ since then $E^\vI_{\uv,\sIR}=0$. The rest of the proof is by induction.

\textit{The base case:} 
Consider a list of indices $\vI$ such that \mbox{$\rri(\vI)=0$}. In this case the cumulants $E^\vI_{\uv,\sIR}$ coincide with the cumulants of the white noise $\xi$. The only non-vanishing cumulant is the covariance corresponding to $\rn(\vI)=2$, $\mathsf{m}(\vI)=(0,0)$, $\rrm(\vI)=0$, $\ra(\vI)=0$, $\rr(\vI)=0$ and $\rs(\vI)=0$. It holds that
\begin{equation}
 \|\llangle K_\sIR^{\ast\oo}\ast\xi,K_\sIR^{\ast\oo}\ast\xi\rrangle\|_{\cV^{\mathsf{m}}} \leq \sup_{x_1\in\bT}\int_{\bT}|\bE(\xi(x_1)\xi(\rd x_2))| = 1.
\end{equation}
This finishes the proof of the base case.

\textit{Induction step:}
Fix $i\in\bN_+$ and $m\in\bN_0$. Assume that the theorem is true for all lists of indices $\vI$ such that either $\rri(\vI)<i$, or $\rri(\vI)=i$ and $\rrm(\vI)>m$. We shall prove the theorem for all $\vI$ such that $\rri(\vI)=i$ and $\rrm(\vI)=m$. 

Consider the case $\rs(\vI)> 0$. In this case by Lemma~\ref{lem:flow_E_form_bound}~(A) and (B) the cumulants $E^\vI_{\uv,\sIR}$ can be expressed in terms of the cumulants for which the statement of the theorem has already been established. As a result, the bound~\eqref{eq:thm_cumulants} with $\rs(\vI)> 0$ follows from the inductive assumption and Lemma~\ref{lem:flow_E_form_bound}~(C).

Now consider $\vI = ((i_1,m_1,a_1,0,r_1),\ldots,(i_n,m_n,a_n,0,r_n))$, $\rs(\vI)=0$. It follows from Def.~\ref{dfn:cumulants_eff_force} of the cumulants $E^\vI_{\uv,\sIR}$ that 
\begin{equation}\label{eq:thm_cumulants_ind_step}
 E^\vI_{\uv,\sIR}
 = 
 E^\vI_{\uv,0} + \sum_{q=1}^n \int_0^\sIR E^{\vI_q}_{\uv,\uIR}\,\rd\uIR,
 \quad
 E^\vI_{\uv,\sIR}
 = 
 E^\vI_{\uv,1/2} - \sum_{q=1}^n \int_\sIR^{1/2} E^{\vI_q}_{\uv,\uIR}\,\rd\uIR,
\end{equation}
where $\vI_q = ((i_1,m_1,a_1,0,r_1),\ldots, (i_q,m_q,a_q,1,r_q),\ldots,(i_n,m_n,a_n,0,r_n))$. Note that $\rs(\vI_q)=1$, hence the bound~\eqref{eq:thm_cumulants} has already been established for $E^{\vI_q}_{\uv,\uIR}$. We will use the first of Eqs.~\eqref{eq:thm_cumulants_ind_step} to bound the irrelevant cumulants $E^\vI_{\uv,\sIR}$, i.e. those with $\vI$ such that $\varrho(\vI)+(\rn(\vI)-1)\rdim>0$. The second equality will be used to bound certain contributions to the relevant cumulants $E^\vI_{\uv,\sIR}$, i.e. those with $\vI$ such that $\varrho(\vI)+(\rn(\vI)-1)\rdim\leq0$.

First, let us analyze the irrelevant contributions. If $\rn(\vI)>1$, then $E^\vI_{\uv,0}$ is a joint cumulant of a list of at least two random distributions. Since $\rri(\vI)=i>0$ one of the elements of this list is a deterministic distribution of the form $\partial^r_\uv F^{i,m,a}_{\uv,0}$. Hence, the cumulant vanishes. If $\rn(\vI)=1$ and $\varrho(\vI)>0$, then $E^\vI_{\uv,0}$ coincides with $\partial^r_\uv F^{i,m,a}_{\uv,0}=\partial^r_\uv F^{i,m,a}_{\uv\phantom{0}}$ for some $i,m,a$ such that $\varrho(i,m)+|a|>0$ and consequently vanishes. To prove the bound for $E^\vI_{\uv,\mu}$ we use the first of Eqs.~\eqref{eq:thm_cumulants_ind_step}. As we argued above, the first term on the RHS of this equation vanishes. The claim of the theorem is a consequence of the bound
\begin{equation}\label{eq:K_bound_s_u}
 \|K_\sIR^{\ast\oo,\otimes(n+m)}\ast E^\vI_{\uv,\sIR}\|_{\cV^{\mathsf{m}}}\leq 
 \sum_{q=1}^n\int_0^\sIR \|K_\uIR^{\ast\oo,\otimes(n+m)}\ast E^{\vI_q}_{\uv,\uIR}\|_{\cV^{\mathsf{m}}}\,\rd\uIR.
\end{equation}
We used the fact that $\|K_\sIR^{\ast\oo,\otimes(n+m)}\ast E^{\vI_q}_{\uv,\uIR}\|_{\cV^{\mathsf{m}}}\leq \|K_\uIR^{\ast\oo,\otimes(n+m)}\ast E^{\vI_q}_{\uv,\uIR}\|_{\cV^{\mathsf{m}}}$ for $\uIR\leq\sIR$ which follows from  Lemma~\ref{lem:kernel_u_v} and Remark~\ref{rem:Vm_K2}.

Next, let us analyze the relevant contributions. We note that the inequality $\varrho(\vI)+(\rn(\vI)-1)\rdim\leq0$ implies \mbox{$\rn(\vI)=1$}. Consequently, $\vI=(i,m,a,0,r)$ for some $r\in\{0,1,2\}$ and $a\in\frM$ such that $\varrho(i,m)+|a|\leq0$. Hence, $E^\vI_{\uv,\sIR}=\bE\,\partial^r_\uv F^{i,m,a}_{\uv,\sIR}$. We first study
\begin{equation}
 \bE\, \partial_\uv^r f^{i,m,a}_{\uv,\sIR}
 =\fI( \bE\,\partial_\uv^r F^{i,m,a}_{\uv,\sIR})=
 \fI(K_\sIR^{\ast\oo,\otimes(1+m)}\ast E^\vI_{\uv,\sIR})\in\bR,
\end{equation}
where the map $\fI$ was introduced in Def.~\ref{dfn:map_I}. Note that by the translational invariance $\bE\, f^{i,m,a}_{\uv,\sIR}$ is a constant function. The application of the map $\fI$ to both sides of the second equation in~\eqref{eq:thm_cumulants_ind_step} yields
\begin{equation}\label{eq:thm_cumulants_f_ren}
 \bE\, \partial_\uv^r f^{i,m,a}_{\uv,\sIR}
 =
 \bE\, \partial_\uv^r f^{i,m,a}_{\uv,1/2}
 -
 \int_\sIR^{1/2} \fI(E^{\vI_1}_{\uv,\uIR} )\,\rd\uIR,
\end{equation}
where $E^{\vI_1}_{\uv,\uIR} = \bE\, \partial_\uIR\partial_\uv^r F^{i,m,a}_{\uv,\uIR}$.  Recalling that $\cV^m\equiv\cV^{\mathsf{m}}$ for $n=1$ and using Lemma~\ref{lem:map_I} we arrive at
\begin{equation}\label{eq:thm_cumulants_f_ren_bound}
 |\bE\, \partial_\uv^r f^{i,m,a}_{\uv,\sIR}|
 \leq
 |\bE\, \partial_\uv^r f^{i,m,a}_{\uv,1/2}|
 +
 \int_\sIR^{1/2} \|K^{\ast\oo,\otimes(1+m)}_{\uIR}\ast E^{\vI_1}_{\uv,\uIR}\|_{\cV^m}\,\rd\uIR.
\end{equation}
By assumption, $\bE\, f^{i,m,a}_{\uv,1/2}$ is independent of $\uv$ and $\bE\, \partial_\uv f^{i,m,a}_{\uv,1/2}=0$. Hence, using the bound~\eqref{eq:thm_cumulants_f_ren_bound} and the bound~\eqref{eq:thm_cumulants} applied to $E^{\vI_1}_{\uv,\sIR}$ we obtain
\begin{multline}\label{eq:thm_cumulants_relevant}
 |\bE\, \partial_\uv^r f^{i,m,a}_{\uv,\sIR}|=|\fI(\bE\,\partial_\uv^r F^{i,m,a}_{\uv,\sIR})|
 \lesssim 1+\int_\sIR^{1/2} [\uv]^{(\varepsilon-\sigma)r}
 [\uIR]^{\varrho_\varepsilon(i,m)+|a|-\sigma}\,\rd\uIR
 \\
 \lesssim 
 [\uv]^{(\varepsilon-\sigma)r}
 [\sIR]^{\varrho_\varepsilon(i,m)+|a|}.
\end{multline}
If $m=0$, then $a=0$ and $E^\vI_{\uv,\sIR} = \bE\, \partial^r_\uv F^{i,0,0}_{\uv,\sIR}=\bE\,\partial_\uv^r f^{i,0,0}_{\uv,\sIR}$. Hence, in this case the statement of the theorem follows from the bound~\eqref{eq:thm_cumulants_relevant}. To prove the case $m>0$ we first recall that we have already proved the bounds
\begin{equation}\label{eq:cumulants_proof_rel1}
 \|K_\sIR^{\ast\oo,\otimes(1+m)}\ast \bE\, \partial_\uv^r F^{i,m,b}_{\uv,\sIR}\|_{\cV^m} \lesssim [\uv]^{(\varepsilon-\sigma)r}
 [\sIR]^{\varrho_\varepsilon(i,m)+|b|}
\end{equation}
for all irrelevant coefficients $F^{i,m,b}_{\uv,\sIR}$ and the bounds
\begin{equation}\label{eq:cumulants_proof_rel2}
 |\fI(\bE\, \partial_\uv^r F^{i,m,b}_{\uv,\sIR})|\lesssim [\uv]^{(\varepsilon-\sigma)r}
 [\sIR]^{\varrho_\varepsilon(i,m)+|b|}
\end{equation}
for all relevant coefficients $F^{i,m,b}_{\uv,\sIR}$. Let $\ro\in\bN_+$ be the smallest positive integer such that $\varrho(i,m)+\ro>0$ and note that by Theorem~\ref{thm:taylor} it holds that
\begin{equation}
 \bE\, \partial_\uv^r F^{i,m,a}_{\uv,\sIR}
 =
 \fX^a_\ro(\fI(\bE\, \partial_\uv^r F^{i,m,b}_{\uv,\sIR}),
 \bE\, \partial_\uv^r F^{i,m,b}_{\uv,\sIR}).
\end{equation}
The arguments of the map $\fX^a_\ro$ above satisfy the bounds~\eqref{eq:cumulants_proof_rel2} and~\eqref{eq:cumulants_proof_rel1}. This together with Theorem~\ref{thm:taylor_bounds} applied with $C\lesssim [\uv]^{(\varepsilon-\sigma)r}[\sIR]^{\varrho_\varepsilon(i,m)}$ implies the statement of the theorem.
\end{proof}

\section*{Acknowledgments}

I wish to thank Markus Tempelmayr and Pavlos Tsatsoulis for discussions and useful comments. The main part of this work was carried out when the author was affiliated with the Max-Planck Institute for Mathematics in the Sciences in Leipzig. The financial support of the Max-Planck Society, grant Proj. Bez. M.FE.A.MATN0003 and partial support of the National Science Centre, Poland, grant `Sonata Bis' 2019/34/E/ST1/00053 are gratefully acknowledged.


\begin{thebibliography}{99}


\bibitem{bruned2021renormalising}
Y. Bruned, A. Chandra, I. Chevyrev, M. Hairer, 
``Renormalising SPDEs in regularity structures,'' 
J. Eur. Math. Soc. {\bf 23}(3), 869--947 (2021)
[arXiv:1711.10239]

\bibitem{bruned2019algebraic}
Y. Bruned, M. Hairer, L. Zambotti, 
``Algebraic renormalisation of regularity structures,''
Invent. Math. {\bf 215}(3), 1039--1156 (2019)
[arXiv:1610.08468]

\bibitem{chandra2016bphz}
A. Chandra, M. Hairer, 
``An analytic BPHZ theorem for regularity structures,''
[arXiv:1612.08138]

\bibitem{duch}
P. Duch,
``Flow equation approach to singular stochastic PDEs,''
[arXiv:2109.11380]


\bibitem{hairer2014structures} 
M. Hairer, 
``A theory of regularity structures,'' 
Invent. Math. {\bf 198}(2), 269--504 (2014)
[arXiv:1303.5113]

\bibitem{kupiainen2016rg}
A. Kupiainen, ``Renormalization group and stochastic PDEs,'' 
Ann. Henri Poincaré {\bf 17}(3), 497--535 (2016)
[arXiv:1410.3094]

\bibitem{kupiainen2017kpz}
A. Kupiainen, M. Marcozzi, 
``Renormalization of generalized KPZ equation,''
J. Stat. Phys. {\bf 166}, 876--902 (2017) 
[arXiv:1604.08712]

\bibitem{muller2003}
V. Müller, 
``Perturbative renormalization by flow equations,''
Rev. Math. Phys. {\bf 15}(05), 491--558 (2003)
[arXiv:hep-th/0208211]

\bibitem{peccati2011wiener}
G. Peccati, M. Taqqu,
``Wiener Chaos: Moments, Cumulants and Diagrams: A survey with computer implementation,''
(Springer, 2011)

\bibitem{polchinski1984}
J. Polchinski, 
``Renormalization and effective lagrangians,'' 
Nuclear Physics B, {\bf 231}(2), 269--295, (1984)

\bibitem{wilson1971}
K. Wilson,
``Renormalization Group and Critical Phenomena. I. Renormalization Group and the Kadanoff Scaling Picture,''
Phys. Rev. B, {\bf 4} 3174--3183, (1971)



\end{thebibliography}
\end{document}